\documentclass[a4paper,12pt]{article}

\usepackage{amsmath,amstext,amscd,amsthm,amsfonts,diagrams}
\usepackage{amssymb}
\newtheorem*{theo}{Theorem}
\newtheorem{theorem}{Theorem}[section]
\newtheorem{lemma}[theorem]{Lemma}

\newtheorem{corollary}[theorem]{Corollary}
\theoremstyle{remark}
\newtheorem{rk}[theorem]{Remark}
\newtheorem{example}[theorem]{Example}

\newcommand{\quot}{\ensuremath{/ \hspace{-1.2mm}/}}

\def\Ad{\mathop{\rm Ad}\nolimits}
\def\ad{\mathop{\rm ad}\nolimits}
\def\Int{\mathop{\rm Int}\nolimits}
\def\Aut{\mathop{\rm Aut}\nolimits}

\def\Lie{\mathop{\rm Lie}\nolimits}
\def\GL{\mathop{\rm GL}\nolimits}
\def\SL{\mathop{\rm SL}\nolimits}

\def\Sp{\mathop{\rm Sp}\nolimits}

\def\charac{\mathop{\rm char}\nolimits}
\def\dim{\mathop{\rm dim}\nolimits}

\def\Hom{\mathop{\rm Hom}\nolimits}

\def\ker{\mathop{\rm ker}\nolimits}
\def\End{\mathop{\rm End}\nolimits}

\def\rank{\mathop{\rm rk}\nolimits}

\def\rank{\mathop{\rm rk}\nolimits}

\def\Spin{\mathop{\rm Spin}\nolimits}

\def\tildeAneventwoDynk{\put(0,0){\circle{4}} \put(0,2){\line(1,0){40}}
\put(0,-2){\line(1,0){40}} \put(15,-5){\line(2,1){10}}
\put(15,5){\line(2,-1){10}} \put(40,0){\circle{4}}
\multiput(42,0)(8,0){10}{\line(1,0){4}} \put(120,0){\circle{4}}
\put(120,2){\line(1,0){40}} \put(120,-2){\line(1,0){40}}
\put(135,5){\line(2,-1){10}} \put(135,-5){\line(2,1){10}}
\put(160,0){\circle{4}}
\put(0,-10){\makebox(0,0){$\scriptstyle{1}$}}
\put(40,-10){\makebox(0,0){$\scriptstyle{2}$}}
\put(120,-10){\makebox(0,0){$\scriptstyle{2}$}}
\put(160,-10){\makebox(0,0){$\scriptstyle{2}$}}}
\def\tildeAtwotwoDynk{\put(0,0){\circle{6}} \put(3,1){\line(1,0){34}} \put(3,-1){\line(1,0){34}}
\put(0,3){\line(1,0){40}} \put(0,-3){\line(1,0){40}}
\put(15,5){\line(2,-1){10}} \put(15,-5){\line(2,1){10}}
\put(40,0){\circle{6}}
\put(0,-10){\makebox(0,0){$\scriptstyle{1}$}}
\put(40,-10){\makebox(0,0){$\scriptstyle{2}$}} }
\def\tildeAnoddtwoDynk{\put(0,0){\circle{4}} \put(0,2){\line(1,0){40}}
\put(0,-2){\line(1,0){40}} \put(15,-5){\line(2,1){10}}
\put(15,5){\line(2,-1){10}} \put(40,0){\circle{4}}
\multiput(42,0)(8,0){10}{\line(1,0){4}} \put(120,0){\circle{4}}
\put(121.5,1.5){\line(1,1){27}} \put(150,30){\circle{4}}
\put(121.5,-1.5){\line(1,-1){27}} \put(150,-30){\circle{4}}
\put(0,-10){\makebox(0,0){$\scriptstyle{1}$}}
\put(40,-10){\makebox(0,0){$\scriptstyle{2}$}}
\put(135,0){\makebox(0,0){$\scriptstyle{2}$}}
\put(165,30){\makebox(0,0){$\scriptstyle{1}$}}
\put(160,-30){\makebox(0,0){$\scriptstyle{1}$}}}
\def\tildeDntwoDynk{\put(0,0){\circle{4}} \put(0,2){\line(1,0){40}}
\put(0,-2){\line(1,0){40}} \put(15,-5){\line(2,1){10}}
\put(15,5){\line(2,-1){10}} \put(40,0){\circle{4}}
\multiput(42,0)(8,0){10}{\line(1,0){4}} \put(120,0){\circle{4}}
\put(120,2){\line(1,0){40}} \put(120,-2){\line(1,0){40}}
\put(135,0){\line(2,-1){10}} \put(135,0){\line(2,1){10}}
\put(160,0){\circle{4}}
\put(0,-10){\makebox(0,0){$\scriptstyle{1}$}}
\put(40,-10){\makebox(0,0){$\scriptstyle{1}$}}
\put(120,-10){\makebox(0,0){$\scriptstyle{1}$}}
\put(160,-10){\makebox(0,0){$\scriptstyle{1}$}}}
\def\DfourthreeDynk{\put(0,0){\circle{4}} \put(2,0){\line(1,0){36}}
\put(0,2){\line(1,0){40}} \put(0,-2){\line(1,0){40}}
\put(15,0){\line(2,-1){10}} \put(15,-0){\line(2,1){10}}
\put(40,0){\circle{4}}
\put(0,-10){\makebox(0,0){$\scriptstyle{2}$}}
\put(40,-10){\makebox(0,0){$\scriptstyle{1}$}} }
\def\tildeDfourthreeDynk
{\put(0,0){\circle{4}}
\put(2,0){\line(1,0){36}} \put(40,0){\DfourthreeDynk}
\put(0,-10){\makebox(0,0){$\scriptstyle{1}$}} }
\def\EsixtwoDynk{\put(0,0){\circle{4}} \put(2,0){\line(1,0){36}}
\put(40,0){\circle{4}} \put(40,2){\line(1,0){40}}
\put(40,-2){\line(1,0){40}} \put(55,0){\line(2,-1){10}}
\put(55,0){\line(2,1){10}} \put(80,0){\circle{4}}
\put(82,0){\line(1,0){36}} \put(120,0){\circle{4}}
\put(0,-10){\makebox(0,0){$\scriptstyle{2}$}}
\put(40,-10){\makebox(0,0){$\scriptstyle{3}$}}
\put(80,-10){\makebox(0,0){$\scriptstyle{2}$}}
\put(120,-10){\makebox(0,0){$\scriptstyle{1}$}} }
\def\tildeEsixtwoDynk{\put(0,0){\circle{4}}
\put(2,0){\line(1,0){36}} \put(40,0){\EsixtwoDynk}
\put(0,-10){\makebox(0,0){$\scriptstyle{1}$}} }

\title{KW-sections for exceptional type Vinberg's $\theta$-groups}
\author{Paul Levy \\
paul.levy@epfl.ch}

\begin{document}

\bibliographystyle{plain}

\maketitle{}{}

\begin{abstract}
Let $k$ be an algebraically closed field of characteristic not equal to 2 or 3, let $G$ be an almost simple algebraic group of type $F_4$, $G_2$ or $D_4$ and let $\theta$ be an automorphism of $G$ of finite order, coprime to the characteristic.
In this paper we consider the $\theta$-group (in the sense of Vinberg) associated to these choices; we classify the positive rank automorphisms via Kac diagrams and we describe the little Weyl group in each case.
As a result we show that all $\theta$-groups in types $G_2$, $F_4$ and $D_4$ have KW-sections, confirming a conjecture of Popov in these cases.
\end{abstract}

\section{Introduction}

Let $G$ be a reductive algebraic group over the algebraically closed field $k$ and let ${\mathfrak g}=\Lie(G)$.
Let $\theta$ be a semisimple automorphism of $G$ of order $m$, let $d\theta$ be the differential of $\theta$ and let $\zeta$ be a primitive $m$-th root of unity in $k$.
(Thus if $k$ is of positive characteristic $p$ then $p\nmid m$.)
There is a direct sum decomposition $${\mathfrak g}={\mathfrak g}(0)\oplus\ldots \oplus{\mathfrak g}(m-1)\;\;\mbox{where}\;{\mathfrak g}(i)=\{ x\in{\mathfrak g}\,|\,d\theta(x)=\zeta^i x\}$$
This is a ${\mathbb Z}/m{\mathbb Z}$-grading of ${\mathfrak g}$, that is $[{\mathfrak g}(i),{\mathfrak g}(j)]\subset{\mathfrak g}(i+j)$ ($i,j\in{\mathbb Z}/m{\mathbb Z}$).
Let $G(0)=(G^\theta)^\circ$.
Then $G(0)$ is reductive, $\Lie(G(0))={\mathfrak g}(0)$ and $G(0)$ stabilizes each component ${\mathfrak g}(i)$.
In \cite{vin}, Vinberg studied invariant-theoretic properties of the $G(0)$-representation ${\mathfrak g}(1)$.
The central concept in \cite{vin} is that of a {\it Cartan subspace}, which is a subspace of ${\mathfrak g}(1)$ which is maximal subject to being commutative and consisting of semisimple elements.
The principal results of \cite{vin} (for $k={\mathbb C}$) are:

 - any two Cartan subspaces of ${\mathfrak g}(1)$ are $G(0)$-conjugate and any semisimple element of ${\mathfrak g}(1)$ is contained in a Cartan subspace.

 - the $G(0)$-orbit through $x\in{\mathfrak g}(1)$ is closed if and only if $x$ is semisimple, and is unstable (that is, its closure contains $0$) if and only if $x$ is nilpotent.

 - let ${\mathfrak c}$ be a Cartan subspace of ${\mathfrak g}(1)$ and let $W_{\mathfrak c}=N_{G(0)}({\mathfrak c})/Z_{G(0)}({\mathfrak c})$, the little Weyl group.
Then we have a version of the Chevalley restriction theorem: the embedding ${\mathfrak c}\hookrightarrow{\mathfrak g}(1)$ induces an isomorphism $k[{\mathfrak g}(1)]^{G(0)}\rightarrow k[{\mathfrak c}]^{W_{\mathfrak c}}$.

 - $W_{\mathfrak c}$ is a finite group generated by complex (often called pseudo-)reflections, hence $k[{\mathfrak c}]^{W_{\mathfrak c}}$ is a polynomial ring.
 
In the case of an involution, the decomposition ${\mathfrak g}={\mathfrak g}(0)\oplus{\mathfrak g}(1)$ is the symmetric space decomposition, much studied since the seminal paper of Kostant and Rallis \cite{kostrall}.
(Many of the results of \cite{kostrall} were generalized to good positive characteristic by the author in \cite{invs}.)
While the theory of $\theta$-groups can in some ways be thought of as an extension of the theory of symmetric spaces, there are certain differences of emphasis.
Broadly speaking, one can say that the results here on geometry and orbits are weaker than for symmetric spaces, but the connection with groups generated by pseudoreflections is more interesting.
Recall that a KW-section is an affine linear subspace ${\mathfrak v}\subset{\mathfrak g}(1)$ such that restricting to ${\mathfrak v}$ induces an isomorphism $k[{\mathfrak g}(1)]^{G(0)}\rightarrow k[{\mathfrak v}]$.
A long-standing conjecture of Popov \cite{popov} is the existence of a KW-section in ${\mathfrak g}(1)$ for the action of $G(0)$.
In characteristic zero, this conjecture is known to hold in the cases when $G(0)$ is semisimple \cite{pansemislice} and when ${\mathfrak g}(1)$ contains a regular nilpotent element of ${\mathfrak g}$ (the `N-regular' case) \cite{panslice}, both due to Panyushev.

In \cite{thetagroups}, the results of \cite{vin} were extended to the case where $k$ has positive characteristic $p$ and $G$ satisfies the {\it standard hypotheses}: (A) the derived subgroup $G'$ of $G$ is simply-connected, (B) $p$ is good for $G$ and (C) there exists a non-degenerate $G$-equivariant symmetric bilinear form $\kappa:{\mathfrak g}\times{\mathfrak g}\rightarrow k$.
(In fact, the first three results mentioned above hold for all $p>2$; the standard hypotheses are only required for the fourth.)
Moreover, an analysis of the little Weyl group and an extension of Panyushev's result on N-regular automorphisms revealed that KW-sections exist for all classical graded Lie algebras in zero or odd positive characteristic.
This leaves Popov's conjecture open in the following cases: (i) $G$ is of exceptional type, and (ii) $G$ is simply-connected of type $D_4$, $p>3$ and $\theta$ is an outer automorphism of $G$ such that $\theta^3$ is inner.
Following Vinberg, we refer to all such cases as exceptional type $\theta$-groups.

In characteristic zero the automorphisms of finite order were classified by Kac \cite{kacautos}.
In Sect. 2 we show that Kac's classification is valid in characteristic $p$, restricting to automorphisms of order coprime to $p$.
(The extension of Kac's results to positive characteristic was already outlined by Serre \cite{serrekac}.)
Subsequently, we determine (Sect. 4) the positive rank exceptional $\theta$-groups of types $F_4$, $G_2$ and $D_4$ and describe (Sect. 5) the corresponding little Weyl groups.
(The rank of an automorphism is the dimension of a Cartan subspace, thus $\theta$ is of positive rank if and only if the invariants are non-trivial.)
Our method is a continuation of the method used in \cite{thetagroups} to determine the Weyl group for the classical graded Lie algebras in positive characteristic.
In particular, given an automorphism $\theta$ let $T$ be a $\theta$-stable maximal torus such that $\Lie(T)$ contains a Cartan subspace.
Then $\theta=\Int n_w$ where $n_w\in N_{\Aut G}(T)$, and $w=n_wT\in N_{\Aut G}(T)/T$ is either of order $m$, or is trivial (in which case $\theta$ is of zero rank).
Using Carter's classification of conjugacy classes in the Weyl group, this approach gives us a relatively straightforward means to determine the positive rank automorphisms and their Weyl groups (see Tables 1-3).
This classification allows us to show that all $\theta$-groups in types $G_2$, $F_4$ or $D_4$ have KW-sections (see Thm. \ref{main}):
%From this classification, we deduce our main result (see Thm. \ref{main}):

\begin{theo}
Any $\theta$-group of type $G_2$, $F_4$ or $D_4$ admits a KW-section.
\end{theo}

Together with \cite{thetagroups}, this leaves only the $\theta$-groups of type $E$ to deal with.
These will be dealt with in subsequent work with B. Gross, M. Reeder and J-K. Yu.
%The methods employed in this paper could be applied to solve cases $E_6$, $E_7$ and $E_8$, but the calculations would be rather laborious.
%In work to appear with B. Gross, M. Reeder and J-K. Yu, we will determine all of the remaining (exceptional type) positive rank $\theta$-groups and %establish Thm. \ref{main} for each of them, thus completing the proof of Popov's conjecture.

{\it Notation.}

Throughout, $G$ will denote an almost simple (semisimple) algebraic group and ${\mathfrak g}$ its Lie algebra.
If $T$ is a maximal torus of $G$ and $\alpha\in\Phi(G,T)$ then we denote the corresponding coroot by $\alpha^\vee$, which we can also consider as a cocharacter $k^\times\rightarrow T$.
If $H$ is a subgroup of $G$ then $N_G(H)$ (resp. $Z_G(H)$) will denote its normalizer (resp. centralizer) in $G$.
Similar notation will be used for elements of $G$ and subalgebras or elements of ${\mathfrak g}$.
The connected component of an algebraic group $H$ is denoted $H^\circ$; the derived subgroup of a connected algebraic group is denoted $H'$.
We denote by $\mu_s$ the cyclic group of order $s$.
%($s$ will always be coprime to $m$.)
Throughout the paper we will assume that $m$ is coprime to the characteristic of $k$, if this is positive.

{\it Acknowledgements.}

I would like to thank Ross Lawther for directing me towards the paper of Carter.
I have also benefited from numerous helpful comments from Dmitri Panyushev, Oksana Yakimova, Willem de Graaf, Jiu-Kang Yu, Benedict Gross, Mark Reeder and Gunter Malle.

\section{Preparation}

We continue with the basic set-up of the introduction.
Here we recall some results and definitions which will be necessary in what follows.

Let $\Phi$ be an irreducible root system with basis $\Delta=\{\alpha_1,\ldots ,\alpha_n\}$.
Let $\hat\alpha=\sum_{i=1}^n m_i\alpha_i$ be the longest root with respect to $\Delta$.
Recall that $p$ is {\it good} for $G$ if $p>m_i$ for all $i$, $1\leq i\leq n$; otherwise $p$ is bad.
Specifically, $p$ is bad if either: $p=2$ and $\Phi$ is not of type $A$; $p=3$ and $\Phi$ is of exceptional type, or; $p=5$ and $\Phi$ is of type $E_8$.
If $\Phi$ is a reducible root system, then $p$ is good for $\Phi$ if and only if $p$ is good for each irreducible component of $\Phi$; $p$ is good for the reductive group $G$ if and only if it is good for the root system of $G$.
We will sometimes refer to the {\it standard hypotheses} on a group $G$: (A) that $p=\charac k$ is good for $G$; (B) that $G'$ is simply-connected; (C) that there exists a non-degenerate symmetric bilinear form ${\mathfrak g}\times{\mathfrak g}\rightarrow k$.
If $G$ is simple and of exceptional type and $p$ is good for $G$, then $G$ is separably isogenous to a group satisfying the standard hypotheses.
(In particular, if $G$ is of type $F_4$ or $G_2$ then (B) and (C) are automatic as long as (A) is satisfied.)

Let $T$ be any $\theta$-stable torus in $G$.
If $\theta$ is an involution then it is not difficult to see that there is a decomposition $T=T_+\cdot T_-$, where $T_+=\{Êt\in T\, |\, \theta(t)=t\}^\circ$ and $T_-=\{ t\in T\, |\, \theta(t)=t^{-1}\}^\circ$.
In \cite{thetagroups} we introduced analogues of $T_\pm$ for an automorphism of arbitrary finite order $m$ (if the characteristic is positive, then we require $p\nmid m$).
Let $p_d(x)$ be the (monic) minimal polynomial of $e^{2\pi i/d}$ over ${\mathbb Q}$.
Since $p_d$ has coefficients in ${\mathbb Z}$, we can (and will) consider it as a polynomial in ${\mathbb F}_p[x]$ as well.
If $p\nmid d$ then $p_d(x)$ has no repeated roots in $k$.
For any polynomial $f=\sum_{i=0}^na_i x^i\in{\mathbb Z}[x]$, let $\overline{f}(\theta)$ denote the algebraic endomorphism $T\rightarrow T$, $t\mapsto\prod_{i=0}^n \theta^i(t^{a_i})$.
Then the map ${\mathbb Z}[x]\rightarrow \End T$, $f\mapsto \overline{f}(\theta)$ is a homomorphism of rings, where the addition in $\End T$ is pointwise multiplication of endomorphisms, and the multiplication is composition.

\begin{lemma}[{\cite[Lemma 1.10]{thetagroups}}]\label{stabletori}
For each $d|m$, let $T_d=\{ t\in T\, |\, \overline{p_{m/d}}(\theta)(t)=e\}^\circ$.
Then $T=\prod_{d|m}T_d$ and $T_{d_1}\cap T_{d_2}$ is finite for any distinct $d_1,d_2$.
Moreover, ${\mathfrak t}=\sum_{i=0}^m{\mathfrak t}(i)$ where ${\mathfrak t}(i)={\mathfrak t}\cap{\mathfrak g}(i)$ and $\Lie(T_d)=\sum_{(i,m)=d}{\mathfrak t}(i)$.
\end{lemma}

\begin{rk}
An alternative description of the tori $T_d$ is given in \cite{broue-malle-michel}.
Since $\theta$ is an algebraic automorphism of $T$, it induces an automorphism $\theta^*$ of $Y(T):=\Hom(k^\times, T)$ given by $\theta^*(\lambda)(t)=\theta^*(\lambda(t))$.
Then set $Y(T_d)=\ker p_d(\theta^*)$.
This also allows us to define the tori $T_d$ when $m$ is divisible by $p$; however, in this case we do not have a direct sum decomposition ${\mathfrak t}=\oplus_{d|m}\Lie(T_d)$.
\end{rk}

An immediate application of Lemma \ref{stabletori} is the following.

\begin{lemma}\label{fpfree}
Let $T$ be a maximal torus of $G$ and let $w\in W=N_G(T)/T$ be of order $m$ ($p\nmid m$) and have finitely many fixed points on $T$.
Then any two representatives for $w$ in $N_G(T)$ are $T$-conjugate, and have the same $m$-th power.
\end{lemma}

\begin{proof}
Suppose $n_w\in N_G(T)$ is such that $w=n_wT$.
We have $T=\prod_{i|m,i\neq m}T_i$ and therefore $tw(t)w^2(t)\ldots w^{-1}(t)=1$ for all $t\in T$.
Thus $(n_w t)^m=n_w^m$ for all $t\in T$.
Moreover, since $T_m=\{ 1\}$, the map $T\rightarrow T$, $t\mapsto tw(t^{-1})$ is surjective.
In particular, if $t\in T$ then there exists $s\in T$ such that $tn_w=sn_ws^{-1}$.
\end{proof}

Let now $\theta$ be an automorphism of $G$ of order $m$, let $\zeta$ be a primitive $m$-th root of 1 in $k$ and let ${\mathfrak g}(i)=\{ x\in{\mathfrak g}\, |\, d\theta(x)=\zeta^i x\}$.
A $\theta$-stable torus in $G$ is {\bf $\theta$-split} if $T=T_1$, and is {\bf $\theta$-anisotropic} if $T_m=\{ e\}$, that is, if $\theta$ has finitely many fixed points on $T$.
We recall that a Cartan subspace of ${\mathfrak g}(1)$ is a maximal subspace which is commutative and consists of semisimple elements.
We have the following results on Cartan subspaces.

\begin{lemma}\label{stabtori}
(a) Any two Cartan subspaces of ${\mathfrak g}(1)$ are $G(0)$-conjugate, and any semisimple element of ${\mathfrak g}(1)$ is contained in a Cartan subspace.

(b) Let ${\mathfrak c}$ be a Cartan subspace of ${\mathfrak g}(1)$.
There is a $\theta$-split torus $T_1$ in $G$ such that ${\mathfrak c}\subset \Lie(T_1)$.
We have $\dim T_1=\dim{\mathfrak c}\cdot\varphi(m)$, where $\varphi(m)$ is the Euler number of $m$.
(In particular, $\dim (\Lie(T_1)\cap {\mathfrak g}(i))=\dim{\mathfrak c}$ for any $i$ coprime to $m$.)
If $p>2$, then $T_1$ is unique.

(c) Let ${\mathfrak c}$ be a Cartan subspace of ${\mathfrak g}(1)$ and let $W_{\mathfrak c}=N_{G(0)}({\mathfrak c})/Z_{G(0)}({\mathfrak c})$, the little Weyl group. If $\charac k\neq 2$, then the embedding ${\mathfrak c}\hookrightarrow{\mathfrak g}(1)$ induces an isomorphism $k[{\mathfrak g}(1)]^{G(0)}\rightarrow k[{\mathfrak c}]^{W_{\mathfrak c}}$.

(d) If $\charac k=0$ or if $G$ satisfies the standard hypotheses, then $W_{\mathfrak c}$ is generated by pseudoreflections and $k[{\mathfrak c}]^{W_{\mathfrak c}}$ is a polynomial ring.
\end{lemma}

\begin{proof}
Part (a) is \cite[Thm. 1]{vin}, and \cite[Thm. 2.5 and Cor. 2.6]{thetagroups}.
For (b), see \cite[\S 3.1]{vin} and \cite[Lemma 2.7]{thetagroups}.
Part (c) was \cite[Th. 7]{vin}, \cite[Thm. 2.20]{thetagroups}.
Finally, part (d) is Thm. 8 in \cite{vin} and \cite[Prop. 4.22]{thetagroups}.
\end{proof}

We recall the following results of Steinberg \cite[7.5,9.16]{steinberg}.

\begin{lemma}
(a) If $\theta$ is a semisimple automorphism of $G$ then there is a $\theta$-stable Borel subgroup $B$ of $G$ and a $\theta$-stable maximal torus $T$ of $B$.

(b) Let $\pi:\hat{G}\rightarrow G$ be the universal covering of $G$.
Then there exists a unique automorphism $\hat\theta$ of $\hat{G}$ such that $\pi(\hat\theta(g))=\theta(\pi(g))$ for all $g\in \hat{G}$.
Moreover, if $\theta$ is of order $m$ then so is $\hat\theta$.
\end{lemma}

%Let $G^\theta_Z=\{ g\in G\, \mid\,g^{-1}\theta(g)\in Z(G)\}$.
%Denote by $W^Z_{\mathfrak c}$ the group $N_{G_Z^\theta}({\mathfrak c})/Z_{G_Z^\theta}({\mathfrak c})$.

Let $G$ be a simple, simply-connected group and let ${\mathfrak g}=\Lie(G)$.
Denote by $\Aut G$ (resp. $\Aut{\mathfrak g}$) the algebraic group of rational (resp. restricted Lie algebra) automorphisms of $G$ (resp. ${\mathfrak g}$).
In characteristic zero it is well-known that all automorphisms of ${\mathfrak g}$ arise as differentials of automorphisms of $G$, and the corresponding map $\Aut G\rightarrow\Aut{\mathfrak g}$ is bijective.
This may fail to be true in small characteristic \cite{hogeweij}.

\begin{lemma}
If $p>2$ then differentiation $d:\Aut G\rightarrow\Aut{\mathfrak g}$ is bijective.
\end{lemma}

\begin{proof}
If $G=\SL(n,k)$ then this was proved in \cite[Lemma 1.4]{invs}.
But if $p>2$ then the automorphism group of ${\mathfrak g}$ as an abstract Lie algebra is isomorphic to $\Int G\rtimes{\cal D}$, where ${\cal D}$ is the group of `graph automorphisms' of the Dynkin diagram of $G$ \cite[II. Table 1]{hogeweij}.
Thus any automorphism of ${\mathfrak g}$ is the differential of an automorphism of $G$, and the map $d:\Aut G\rightarrow\Aut{\mathfrak g}$ is bijective.
\end{proof}

%We will also require the following straightforward result.

%\begin{lemma}
%Suppose $G$ is simple and simply-connected.
%If $\Phi_1$ is a root subsystem of $\Phi(G,T)$ and $\Delta_1$ is a basis of $\Phi_1$, then the elements $h_\alpha$ with $\alpha\in\Delta_1$ are linearly independent.
%\end{lemma}

\section{Kac diagrams}

Here we recall Kac's classification \cite{kacautos} of periodic automorphisms of simple Lie algebras in characteristic zero, and show that the classification extends to positive characteristic (considering only those automorphisms of $\Lie(G)$ which arise as differentials of rational automorphisms of $G$, and only automorphisms of order coprime to the characteristic).
For more details on the classification, we recommend chapters 6 to 8 of Kac's book \cite{kacbook}.
Let $G$ be almost simple and simply-connected,  let $\Delta=\{Ê\alpha_1,\ldots ,\alpha_r\}$ be a basis of the root system $\Phi(G,T)$ and let $\tilde\Delta=\Delta\cup\{\alpha_0\}$ be a basis of the extended root system.
Let $\{ h_{\alpha},e_\beta\, |\, \alpha\in\Delta, \beta\in\Phi\}$ be a Chevalley basis for $G$.
%(In fact if $k$ is of positive characteristic then the $h_\alpha$, $\alpha\in\Delta$ aren't necessarily linearly independent, see \cite[Rk. 1.3]{thetagroups}.
%In what follows below this won't be a problem, but we stress that the assumption that the Chevalley elements $h_\alpha$, $\alpha\in\Delta$ are linearly independent.
%Since we will only be considering automorphisms of ${\mathfrak g}$ which arise as differentials of automorphisms of $G$, this will .)
Recall that $\Aut G$ is generated over $\Int G$ by finitely many (graph) automorphisms, in particular $\Aut G/\Int G$ is finite.
If $\theta$ is a (rational) automorphism of $G$ then the {\it index} of $\theta$ is the order of $\theta$ modulo the inner automorphisms.

For the time being, let $k={\mathbb C}$ and let $\zeta=e^{2\pi i/m}$.
To each automorphism $\sigma$ of ${\mathfrak g}$ of (finite) order $m$, Kac associated an infinite-dimensional Lie algebra, the {\it twisted loop algebra}:
$$L({\mathfrak g},\sigma)=\sum_{j\in{\mathbb Z}}t^i\otimes{\mathfrak g}_j\subset{\mathbb C}[t,t^{-1}]\otimes{\mathfrak g}$$ where ${\mathfrak g}_j=\{ x\in{\mathfrak g}\, |\, \sigma(x)=\zeta^j x\}$.
Kac showed that twisted loop algebras are closely related to affine type Lie algebras.
To explain this, let $\widehat{\mathfrak g}^{(s)}$ denote the affine type Lie algebra associated to ${\mathfrak g}$ and $s$ (for $s$ the index of some automorphism of ${\mathfrak g}$): the Dynkin diagram of $\widehat{\mathfrak g}^{(1)}$ is just the extended Dynkin diagram of ${\mathfrak g}$; the Dynkin diagrams for the ``twisted'' types $s>1$ are given below.
It is well known that $\widehat{\mathfrak z}^{(s)}={\mathfrak z}(\widehat{\mathfrak g}^{(s)})$ is one-dimensional and contained in the derived subalgebra $(\widehat{\mathfrak g}^{(s)})'$ (which is of codimension 1 in $\widehat{\mathfrak g}^{(s)}$).
Identify a basis of simple roots for $\widehat{\mathfrak g}^{(1)}$ with $\tilde\Delta$.
Let $\widehat\Phi$ denote the root system of $\widehat{\mathfrak g}^{(s)}$.

\begin{theorem}
a) If $\sigma$ and $\theta$ are two periodic automorphisms of ${\mathfrak g}$ of the same index, then $L({\mathfrak g},\sigma)\cong L({\mathfrak g},\theta)$.
In particular, if $\sigma$ is inner then $L({\mathfrak g},\sigma)\cong L({\mathfrak g},1)={\mathbb C}[t,t^{-1}]\otimes{\mathfrak g}$.

b) $L({\mathfrak g},\sigma)\cong (\widehat{\mathfrak g}^{(s)})'/\widehat{\mathfrak z}^{(s)}$, where $s$ is the index of $\sigma$.

c) The grading of $L({\mathfrak g},\sigma)$ given by $L({\mathfrak g},\sigma)_i=t^i\otimes{\mathfrak g}_i$ induces a grading of $\widehat{\mathfrak g}^{(s)}$.
There is a Cartan subalgebra $\widehat{\mathfrak h}$ of $\widehat{\mathfrak g}^{(s)}$, a basis $\Pi=\{\beta_0,\ldots ,\beta_l\}$ of simple roots in $\widehat\Phi$ and a sequence $(n_0,\ldots ,n_l)$ of non-negative integers such that the grading of $\widehat{\mathfrak g}^{(s)}$ is the grading obtained by putting $\widehat{\mathfrak h}$ in degree zero and $\widehat{\mathfrak g}^{(s)}_{\pm\beta_i}$ in degree $\pm n_i$.

d) In the notation of (c), the order of $\sigma$ is $m=s(\sum_{j=0}^l a_jn_j)$, where $s\delta=s\sum_{j=0}^l a_j\beta_j$ is the shortest positive imaginary root in $\widehat\Phi$.

e) Given a grading of $\widehat{\mathfrak g}^{(s)}$ determined by a sequence $(n_0,\ldots ,n_l)$ of non-negative integers as in (c) there is an induced ${\mathbb Z}/m{\mathbb Z}$-grading of ${\mathfrak g}$ given by the isomorphism ${\mathfrak g}\cong L({\mathfrak g},\sigma)/(t^m,t^{-m})L({\mathfrak g},\sigma)$.%\cong (\widehat{\mathfrak g}^{(s)})'/(\widehat{\mathfrak z}\oplus\sum_{j\neq 0}\widehat{\mathfrak g}^{(s)}_{js\delta})$.
\end{theorem}

\begin{rk}\label{fundrk}
a) Another way to express (c) is that the grading of ${\mathbb Z}\Pi$, considered as an element of the affine apartment, is conjugate by an element of the affine Weyl group to a point in either the fundamental alcove or its opposite.

b) It is easy to show that $\widehat{\mathfrak z}=\{ h\in\widehat{\mathfrak h}\,|\langle\alpha,h\rangle=0\;\mbox{for all}\;\alpha\in\widehat\Phi\}$.
Dually, the space $\{ \chi\in\widehat{\mathfrak h}^*\, |\, \langle\chi,\alpha^\vee\rangle=0\;\mbox{for all}\;\alpha\in\widehat\Phi\}$ is one-dimensional, generated by $\delta$.
If $s=1$ then $\delta=\alpha_0+\sum a_i\alpha_i$, where $\sum a_i\alpha_i=\hat\alpha$ is the highest root of the (finite) root system $\Phi$.
For $s\neq 1$ we mark the coefficients $a_j$ on the diagrams below.
\end{rk}

%One endows ${\mathfrak G}$ with a structure of $k[u,u^{-1}]$ algebra by: $u\cdot x=t^m x$.
%Then Kac showed that the $k[u,u^{-1}]$-algebra ${\mathfrak G}$ (without its ${\mathbb Z}$-grading) depends only on the type of ${\mathfrak g}$ and the index of $\theta$.
%In fact, ${\mathfrak G}$ is the derived subalgebra of a Kac-Moody Lie algebra of affine type: if $\theta$ is an inner automorphism then ${\mathfrak G}$ is the Lie algebra associated to the extended Dynkin diagram of $G$, while if $\theta$ is outer then ${\mathfrak G}$ is (the derived subalgebra of) one of the ``twisted'' Kac-Moody algebras, with diagrams as follows (the index is indicated in parentheses in the superscript, thus $\tilde{A}_{2n}^{(2)}$ denotes the diagram obtained for an outer automorphism in type $A_{2n}$):
The ``twisted type'' Dynkin diagrams are as follows (type $\tilde{A}_3^{(2)}$ is given by $\tilde{D}_3^{(2)}$):

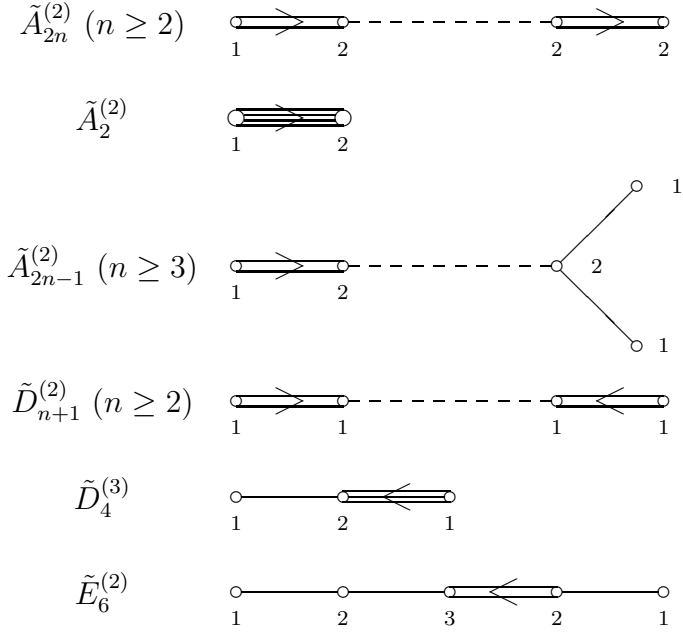
\begin{center}
\begin{picture}(200,35)
\put(0,20){\tildeAneventwoDynk}\put(-50,20){\makebox(0,0){$\tilde A_{2n}^{(2)}$ ($n\geq 2$)}}%
\end{picture}

\begin{picture}(200,35)
\put(0,20){\tildeAtwotwoDynk}\put(-50,20){\makebox(0,0){$\tilde A_2^{(2)}$}}%
\end{picture}

\begin{picture}(200,70)
\put(0,35){\tildeAnoddtwoDynk}\put(-50,35){\makebox(0,0){$\tilde A_{2n-1}^{(2)}$ ($n\geq 3$)}}%
\end{picture}

\begin{picture}(200,35)
\put(0,20){\tildeDntwoDynk}\put(-50,20){\makebox(0,0){$\tilde D_{n+1}^{(2)}$ ($n\geq 2$)}}%
\end{picture}

\begin{picture}(200,35)
\put(0,20){\tildeDfourthreeDynk}\put(-50,20){\makebox(0,0){$\tilde D_{4}^{(3)}$}}%
\end{picture}

\begin{picture}(200,35)
\put(0,20){\tildeEsixtwoDynk}\put(-50,20){\makebox(0,0){$\tilde E_{6}^{(2)}$}}%
\end{picture}
\end{center}

(The diagrams for $\tilde{A}_{2n}^{(2)}$ are not given in their usual orientation.
This choice of orientation is due to our `type-free' choice of graph automorphism, see the discussion below.)

The information $\Pi,(n_0,\ldots ,n_l)$ can be written in a diagram called a Kac diagram $\mu$, which is a copy of the Dynkin diagram of $\widehat{\mathfrak g}^{(s)}$ with the coefficient $n_i$ written above the vertex $\beta_i$.
The order of a Kac diagram $\mu$ is the sum $s(a_0n_0+\ldots +a_ln_l)$.
From now on, any Kac diagram $\mu$ will be primitive, that is, it won't be possible to write $\mu$ as $j\mu'$ for $j>1$ and $\mu'$ some other Kac diagram.
Then to each Kac diagram we can associate an automorphism (which we call a {\it Kac automorphism}) of ${\mathfrak g}$ of order $m$; Kac's theorem tells us that any automorphism of ${\mathfrak g}$ of order $m$ is conjugate (possibly by an element of the outer automorphism group of ${\mathfrak g}$) to a Kac automorphism of order $m$.

In the inner case it is easy to see what the Kac automorphism corresponding to $\mu$ is: it's $\Ad t$, where $t\in T$ is such that $\alpha_i(t)=\zeta^{\mu(\alpha_i)}$ for $1\leq i\leq r$.
(Here $\zeta=e^{2\pi i/m}$.)
Then it is immediate that $\hat\alpha(t)=\zeta^{-\mu(\alpha_0)}$ and that ${\mathfrak g}^t$ is generated by ${\mathfrak t}$ and all ${\mathfrak g}_{\pm\alpha_i}$, where $\mu(\alpha_i)=0$ ($0\leq i\leq r$).

\begin{example}
We describe all (classes of) automorphisms of order 3 of a Lie algebra of type $F_4$.
Here any automorphism is inner.
The possible Kac diagrams of order 3 are $00100$, $11000$ and $10001$.
Thus there are three classes of automorphisms of order 3, with representatives $\Int t_1$, $\Int t_2$, $\Int t_3$, $t_1,t_2,t_3\in T$ satisfying:

(i) $\alpha_1(t_1)=\alpha_3(t_1)=\alpha_4(t_1)=1$, $\alpha_2(t_1)=\zeta$;

(ii) $\alpha_2(t_2)=\alpha_3(t_2)=\alpha_4(t_2)=1$, $\alpha_1(t_2)=\zeta$; and

(iii) $\alpha_1(t_3)=\alpha_2(t_3)=\alpha_3(t_3)=1$, $\alpha_4(t_3)=\zeta$.
\end{example}

We wish to generalise this construction to positive characteristic $p$.
If $p\nmid m$ then we can do this by fixing a primitive $m$-th root $\zeta$ of unity, and continuing exactly as above.
Here, we need to specify that $\Int t$ is an automorphism of $G$ with differential $\Ad t$; while it is no longer true in general that ${\mathfrak g}^t$ is generated by ${\mathfrak t}$ and the subspaces ${\mathfrak g}_{\alpha_i}$ with $\mu(\alpha_i)=0$, it is nevertheless true that the subgroup $Z_G(t)^\circ$ is (reductive and is) generated by $T$ and all root subgroups $U_{\pm\alpha_i}$ (see, for example, \cite[26.3]{hum}) with $\mu(\alpha_i)=0$.

Now consider Kac diagrams of twisted type.
In order to extend Kac's classification to positive characteristic, we will first give a precise description of the Kac automorphism corresponding to each Kac diagram in characteristic zero, and then explain how to extend this description to positive characteristic.
For a given index $s$ let $\gamma$ be a graph automorphism of the root system $\Phi$ which preserves the basis $\Delta$, and by abuse of notation let $\gamma$ also denote the unique automorphism of ${\mathfrak g}$ which satisfies $\gamma(e_\alpha)=e_{\gamma(\alpha)}$ for $\alpha\in\pm\Delta$.
(For each index $s$ there is a unique such graph automorphism of $\Phi$ unless ${\mathfrak g}$ is of type $D_4$ and $s=3$; in this case there are two such $\gamma$, which are mutually inverse and conjugate under the action of the automorphism group of ${\mathfrak g}$.
In this case let us fix the choice satisfying $\gamma(\alpha_1)=\alpha_3$.)
Let $T(0)=(T^\gamma)^\circ$.
For each $\alpha\in\Phi$ let $(\alpha)$ be the set of all $\langle\gamma\rangle$-conjugates of $\alpha$ and let ${\mathfrak g}_{(\alpha)}=\sum_{\beta\in(\alpha)}{\mathfrak g}_\beta$.
Let $\overline\alpha$ be a root which is longest subject to ${\mathfrak g}_{(\alpha)}\not\subseteq {\mathfrak g}^\gamma$.
Specifically, we have: $\overline\alpha=\hat\alpha$ for type $\tilde{A}_{2n}^{(2)}$; $\overline\alpha=\sum_{i=1}^{2n-2}\alpha_i$ for type $\tilde{A}_{2n-1}^{(2)}$; $\overline\alpha=\sum_1^{n-1}\alpha_i$ in type $\tilde{D}_n^{(2)}$; $\overline\alpha=\alpha_1+\alpha_2+\alpha_3$ in type $\tilde{D}_4^{(3)}$; $\overline\alpha=\alpha_1+\alpha_2+\alpha_3+2\alpha_4+2\alpha_5+\alpha_6$ in type $\tilde{E}_6^{(2)}$.
Let $\beta_0$ be the $\langle\gamma\rangle$-conjugacy class of $-\overline\alpha$ and choose $E_{\pm\beta_0}\in{\mathfrak g}_{(\mp \overline\alpha)}$ as follows: in case $\tilde{A}_{2n}^{(2)}$ let $E_{\pm\beta_0}=e_{\mp\overline\alpha}$; in cases $\tilde{A}_{2n-1}^{(2)}$, $\tilde{D}_n^{(2)}$ and $\tilde{E}_6^{(2)}$ let $E_{\pm\beta_0}=e_{\mp\overline\alpha}-\gamma(e_{\mp\overline\alpha})$; in case $\tilde{D}_4^{(3)}$ let $E_{\beta_0}=e_{-\overline\alpha}+\sigma^{-1}\gamma(e_{-\overline\alpha})+\sigma\gamma^2(e_{-\overline\alpha})$ and $E_{-\beta_0}=e_{\overline\alpha}+\sigma\gamma(e_{\overline\alpha})+\sigma^{-1}\gamma^2(e_{\overline\alpha})$, where $\sigma=e^{2\pi i/3}$.
We label the node on the far left hand side (as displayed in the diagrams above; on the top in type $\tilde{A}^{(2)}_{2n-1}$) with $\beta_0$.
(As a character of $T(0)$, $\beta_0$ is just $(-\overline\alpha)|_{T(0)}$.)
Note that $[E_{\beta_0},E_{-\beta_0}]=H_{\beta_0}\in{\mathfrak t}^\gamma$ satisfies $[H_{\beta_0},E_{\pm\beta_0}]=\pm 2E_{\pm\beta_0}$.
The remaining nodes are all labelled by $\langle\gamma\rangle$-conjugacy classes in $\Delta$.
Specifically, we label the nodes $\beta_1,\ldots$ from left to right:

 - $\beta_i=\{\alpha_i,\alpha_{2n+1-i}\}$ for each $i$, $1\leq i\leq n$ in type $\tilde{A}_{2n}^{(2)}$, $n\geq 1$,
 
 - $\beta_i=\{\alpha_i,\alpha_{2n-i}\}$ for $1\leq i\leq n-1$ ($\beta_1$ appearing underneath $\beta_0$) and $\beta_n=\{\alpha_n\}$ in type $\tilde{A}_{2n-1}^{(2)}$,

 - $\beta_i=\{\alpha_i\}$ for $1\leq i\leq n-1$ and $\beta_n=\{\alpha_n,\alpha_{n+1}\}$ in type $\tilde{D}_{n+1}^{(2)}$,

 - $\beta_1=\{\alpha_1,\alpha_3,\alpha_4\}$ and $\beta_2=\{\alpha_2\}$ in type $\tilde{D}_4^{(3)}$,

 - $\beta_1=\{\alpha_1,\alpha_6\}$, $\beta_2=\{\alpha_3,\alpha_5\}$, $\beta_3=\{\alpha_4\}$ and $\beta_4=\{\alpha_2\}$ in type $\tilde{E_6}^{(2)}$.

We can set $E_{\pm\beta_i}=\sum_{\alpha\in\beta_i}e_{\pm\alpha}$ for each $i$, $1\leq i\leq n$ except for $\beta_n$ in $\tilde{A}_{2n}^{(2)}$ in which case ($p\neq 2$ and) we set $E_{\beta_n}=e_{\alpha_n}+e_{\alpha_{n+1}}$, $E_{-\beta_n}=2e_{-\alpha_n}+2e_{-\alpha_{n+1}}$.
Note that (for $0\leq i\leq n$) $\beta_i$ is a well-defined character on $T(0)$ and $\Ad t(E_{\beta_i})=\beta_i(t) E_{\beta_i}$ for any $t\in T(0)$.
Moreover, setting $H_{\beta_i}=[E_{\beta_i},E_{-\beta_i}]$ and letting $\langle\beta_i,\beta_j\rangle$ be the Cartan numbers determined by the twisted diagram, it is easy to establish the following commutation relations:
$$[E_{\beta_i},E_{-\beta_j}]=\left\{\begin{array}{ll} 0 & \mbox{if $i\neq j$,} \\ H_{\beta_i} & \mbox{if $i=j$,} \end{array}\right. \;[H_{\beta_i},E_{\pm\beta_j}]=\pm\langle\beta_i,\beta_j\rangle E_{\pm\beta_j},\;[E_{\pm\beta_i},E_{\pm\beta_j}]=0\,\,\mbox{if $\langle\beta_i,\beta_j\rangle=0$}$$
and $(\ad E_{\beta_i})^{1-\langle\beta_i,\beta_j\rangle}(E_{\beta_j})=(\ad E_{-\beta_i})^{1-\langle\beta_i,\beta_j\rangle}(E_{-\beta_j})=0$ for $i\neq j$.
In addition, $\gamma(H_{\beta_i})=H_{\beta_i}$ for all $i$ and $\gamma(E_{\pm\beta_i})=E_{\pm\beta_i}$ for $1\leq i\leq n$.
Thus, if $I$ is a proper subset of $\Pi$ then the elements $E_{\pm\beta_i}$, $\beta_i\in I$ and ${\mathfrak t}^\gamma$ together generate a reductive subalgebra of ${\mathfrak g}$ with root system given by the subdiagram of the twisted affine diagram which contains exactly the vertices corresponding to elements of $I$.
With this set-up, let $\mu$ be a Kac diagram of twisted affine type, of order $m$.
(The order of $\mu$ is $s\sum_{i=0}^n \mu(\beta_i)m_i$, where $m_i$ are the coefficients indicated on the twisted diagrams above.)
There exists $t\in T(0)$ such that $\beta_i(t)=\zeta^{\mu(\beta_i)}$ for $1\leq i\leq n$.
Let $\theta_\mu=\Ad t\circ\gamma$: then $\theta_\mu$ is of order $m$, $\theta_\mu(E_{\pm\beta_i})=\zeta^{\pm\mu(\beta_i)}E_{\pm\beta_i}$ for $0\leq i\leq n$ and ${\mathfrak g}^{\theta_\mu}$ is the reductive subalgebra of ${\mathfrak g}$ generated by ${\mathfrak t}^\gamma$ and the $E_{\pm\beta_i}$ such that $\mu(\beta_i)=0$.

We can also repeat the above construction in (coprime) positive characteristic, with the following clarification.
For each $\alpha\in\Phi$ let $U_\alpha$ be the root subgroup of $G$ corresponding to $\alpha$ \cite[26.3]{hum}.
There exists a unique isomorphism (of algebraic groups) $\epsilon_\alpha:{\mathbb G}_a\rightarrow U_\alpha$ such that $t\epsilon_\alpha(x)t^{-1}=\epsilon_\alpha(\alpha(t) x)$ for all $t\in T$, $x\in k$ and $(d\epsilon_\alpha)_0(1)=e_\alpha$ (see for example \cite[Thm. 26.3]{hum}).
By abuse of notation, let $\gamma$ denote the unique automorphism of $G$ such that $\gamma(\epsilon_\alpha(x))=\epsilon_{\gamma(\alpha)}(x)$ for $\alpha\in\pm\Delta$.
Fixing a primitive $m$-th root of unity $\zeta$, we can associate an automorphism $\Int t\circ\gamma$ of $G$ to each Kac diagram $\mu$ (of order coprime to $p$).
As before we will say that the automorphism $\theta_\mu$ constructed in this way is the Kac automorphism associated to the Kac diagram $\mu$.
We note that $\gamma$ is the Kac automorphism corresponding to $\mu$ with $\mu(\beta_i)=\left\{\begin{array}{ll} 1 & \mbox{if $i=0$,}Ê\\ 0 & \mbox{otherwise.}\end{array}\right.$
%Note that in characteristic zero the fixed point subalgebra ${\mathfrak g}^\theta$ for a Kac automorphism $\theta$ corresponding to the Kac diagram $\mu$ is generated by ${\mathfrak t}^\gamma$ and the elements $E_{\pm\beta_i}$, $0\leq i\leq n$ with $\mu(\beta_i)=0$.
%In positive characteristic, we have to be somewhat more careful.
Moreover, for $\beta_i\in\Pi$, either all roots $\alpha\in\beta_i$ are orthogonal and the subgroups $U_{\alpha}$, $\alpha\in\beta_i$ commute, or $\beta_i=\{ \alpha_n,\alpha_{n+1}\}$ in type $\tilde{A}_{2n}^{(2)}$.
For $i\geq 1$, if the roots in $\beta_i$ are orthogonal then let $\epsilon_{\pm\beta_i}:{\mathbb G}_a\rightarrow G$, $\epsilon_{\pm\beta_i}(x)=\prod_{\alpha\in\pm\beta_i}\epsilon_\alpha(x)$ and let $U_{\pm\beta_i}=\epsilon_{\pm\beta_i}({\mathbb G}_a)$ (unipotent subgroups of $G$).
Then $(d\epsilon_{\pm\beta_i})_0(1)=E_{\pm\beta_i}$, where $E_{\pm\beta_i}$ is as defined in the paragraph above.
For the case $\beta_n$ in type $\tilde{A}_{2n}^{(2)}$ we can construct $\epsilon_{\pm\beta_n}$ by restricting to the subgroup of type $A_2$ which contains $U_{\pm\alpha_n}$ and $U_{\pm\alpha_{n+1}}$.
Specifically, $\epsilon_{\beta_n}(x)=\epsilon_{\alpha_n}(x/2)\epsilon_{\alpha_{n+1}}(x)\epsilon_{\alpha_n}(x/2)$ and $\epsilon_{-\beta_n}(x)=\epsilon_{-\alpha_n}(x)\epsilon_{-\alpha_{n+1}}(2x)\epsilon_{-\alpha_n}(x)$ define homomorphisms $\epsilon_{\pm\beta_n}:{\mathbb G}_a\rightarrow G^\gamma$ such that $(d\epsilon_{\pm\beta_n})_0(1)=E_{\pm\beta_n}$, where $E_{\pm\beta_n}$ are defined as in characteristic zero.
Finally, either:

 -  $\overline\alpha$ is $\gamma$-stable, in which case we define $\epsilon_{\pm\beta_0}(x)=\epsilon_{\mp\overline\alpha}(x)$;
 
 - or $s=2$ and $\gamma(\overline\alpha)\neq\alpha$, in which case we set $\epsilon_{\pm\beta_0}(x)=\epsilon_{\mp\overline\alpha}(x)\gamma(\epsilon_{\mp\overline\alpha}(-x))$;
 
 - or $s=3$ and we set $\sigma=\zeta^{m/3}$, $\epsilon_{\beta_0}(x)=\epsilon_{-\overline\alpha}(x)\gamma(\epsilon_{-\overline\alpha}(\sigma^{-1}x))\gamma^{-1}(\epsilon_{-\overline\alpha}(\sigma x))$ and $\epsilon_{-\beta_0}(x)=\epsilon_{\overline\alpha}(x)\gamma(\epsilon_{\overline\alpha}(\sigma x))\gamma^{-1}(\epsilon_{\overline\alpha}(\sigma^{-1} x))$.

We have $\gamma(\epsilon_{\pm\beta_0}(x))=\epsilon_{\pm\beta_0}(-x)$ if $s=2$,  $\gamma(\epsilon_{\pm\beta_0})(x)=\epsilon_{\pm\beta_0}(\sigma^{\pm 1} x)$ if $s=3$ and $(d\epsilon_{\pm\beta_0})_0(1)=E_{\pm\beta_0}$, where $E_{\pm\beta_0}$ are defined as in characteristic zero.
Let $U_{\pm\beta_0}=\epsilon_{\pm\beta_0}({\mathbb G}_a)$.
Then each of the subgroups $U_{\pm\beta_i}$ with $\mu(\beta_i)=0$ is contained in $G(0)=(G^{\theta_\mu})^\circ$.
We recall \cite[8.1]{steinberg} that $G(0)$ is reductive.
Moreover, $(T^{\theta_\mu})^\circ=T(0)$ is regular in $G$ by inspection, and hence is a maximal  torus of $G(0)$.
We claim that $G(0)$ is generated by $T(0)$ and the subgroups $U_{\pm\beta_i}$ with $\mu(\beta_i)=0$.
To see this, we remark first of all that $\dim{\mathfrak g}(0)$ is independent of the characteristic.
For if $\alpha\in\Phi$ then let $l(\alpha)$ be the number of $\gamma$-conjugates of $\alpha$ (here either 1, 2 or 3).
Then ${\mathfrak g}_{(\alpha)}\cap{\mathfrak g}(0)$ is of dimension 1 if $d\theta_\mu^{l(\alpha)}(e_\alpha)=e_\alpha$ and of dimension 0 otherwise.
Since this description is clearly independent of the characteristic of $k$, it follows that the same can be said of the dimension of ${\mathfrak g}(0)$.
Let $\tilde\Phi_I$ be the root system with basis $\{\beta_i\in\Pi:\mu(\beta_i)=0\}$.
Then our assumption that $p$ is coprime to the order of $\theta_\mu$ implies that in all cases $p$ is good for $\tilde\Phi_I$ except for one case in type $\tilde{D_4}^{(3)}$ (the Kac diagram $100$) with $p=2$ and one case in type $\tilde{E}_6^{(2)}$ (the Kac diagram $10000$) with $p=3$.
If $p$ is good then $E_{\pm\beta_i}$, $\mu(\beta_i)=0$ and ${\mathfrak t}^\gamma$ generate a subalgebra of ${\mathfrak g}$ which is isomorphic to the Lie algebra of a reductive group with root system $\tilde\Phi_I$.
Since the dimension is independent of the characteristic, this subalgebra must be all of ${\mathfrak g}(0)$.
Now, since the subgroup of $G$ generated by $T(0)$ and the $U_{\pm\beta_i}$, $\mu(\beta_i)$ is contained in $G(0)$, it must be reductive and have root system $\tilde\Phi_I$.

This leaves only the Kac diagrams $100$ in type $\tilde{D}_4^{(3)}$ and $10000$ in type $\tilde{E}_6^{(2)}$; in both cases the corresponding Kac automorphism is just $\gamma$.
In the first case the roots $\alpha\in\Phi$ fall into two classes: (i) those such that $\gamma(e_\alpha)=e_\alpha$ (specifically $\pm\alpha\in\{\alpha_2,\sum_1^4 \alpha_i, \hat\alpha\}$), (ii) those such that $\gamma(\alpha)\neq \alpha$.
It is thus easy to see that $\dim{\mathfrak g}_{(\alpha)}\cap{\mathfrak g}^\gamma=1$ for all $\alpha$ and, setting $H_{\beta_1}=h_{\alpha_1}+h_{\alpha_3}+h_{\alpha_4}$, $H_{\beta_2}=h_{\alpha_2}$, $E_{(\alpha)}=e_\alpha$ in case (i) and $E_{(\alpha)}=e_\alpha+\gamma(e_\alpha)+\gamma^{-1}(e_\alpha)$ in case two, the set $\{ÊH_{\beta_1},H_{\beta_2},E_{(\alpha)}\, |\, \alpha\in\Phi\}$ is a basis for ${\mathfrak g}^\gamma$ and satisfies the relations of a Chevalley basis for a Lie algebra of type $G_2$.
($H_{\beta_1}$ and $H_{\beta_2}$ are always linearly independent since the centre of ${\mathfrak g}$ is generated by $h_{\alpha_1}+h_{\alpha_3}$ and $h_{\alpha_1}+h_{\alpha_4}$.)
Since $G(0)$ contains $T(0)$ and the subgroups $U_{\pm\beta_1}$, $U_{\pm\beta_2}$, this shows that $G(0)$ is generated by these subgroups and is semisimple of type $G_2$.
In the second case, even though $p=3$ is not good for a root system of type $F_4$, it is nevertheless clear that the Lie subalgebra of ${\mathfrak g}$ generated by ${\mathfrak t}^\gamma$ and the $E_{\pm\beta_i}$, $1\leq i\leq 4$ is isomorphic to the Lie algebra of a semisimple group of type $F_4$ (since it contains all appropriate root subspaces for $T(0)$).
Thus by the same argument as above $G(0)$ is generated by $T(0)$ and the subgroups $U_{\pm\beta_i}$, $1\leq i\leq 4$.

We will now show that any automorphism of order $m$ is $\Aut G$-conjugate to a Kac automorphism.
For inner automorphisms, this result appeared in \cite{serrekac}.
One approach would be to construct the loop algebra $L({\mathfrak g},\sigma)$ as in characteristic zero; but this would (presumably) require the theory of Kac-Moody Lie algebras in positive characteristic.
For inner automorphisms, a more direct proof is suggested by considering all ${\mathbb Z}/m{\mathbb Z}$-gradings of ${\mathfrak g}$, and reducing the problem to the statement about the fundamental alcove for the affine Weyl group \ref{fundrk}(a).
(This was the approach of \cite{serrekac}, which also makes it possible to extend Kac's classification to semisimple inner automorphisms of order $p$, that is, ${\mathbb Z}/p{\mathbb Z}$-gradings for which $\Lie(T)\subset{\mathfrak g}(0)$.)
This argument could also be applied to the outer automorphisms by considering the appropriate Bruhat-Tits building; but here we avoid such technical machinery by expressing the problem in terms of conjugacy classes in $T(0)$.

\begin{lemma}
The number of elements of order $m$ in a torus of rank $n$ is independent of the characteristic of the ground field, assuming this characteristic is either zero or coprime to $m$.
\end{lemma}

\begin{proof}
Let $T$ be a torus of rank $n$.
There are $m^{\dim T}$ elements $t$ of $T$ satisfying $t^m=1$.
Let $\phi_{T}(m)$ denote the number of elements of $T$ of order $m$.
Then $\phi_T(m)=m^{dim T}-\sum_{d| m,d\neq m}\phi_T(d)$.
Since $\phi_T(1)=1$, the lemma follows by induction.
\end{proof}

\begin{lemma}\label{kacgen}
Let $G$ be almost simple and of simply-connected or adjoint type.
Then any automorphism of $G$ of finite order coprime to $p$ is $\Aut G$-conjugate to a Kac automorphism.
\end{lemma}

\begin{proof}
Let $T$ be a maximal torus of $G$.
Let $Z(G)$ be the scheme-theoretic centre of $G$ and let $\overline{G}=G/Z(G)$ be the adjoint group of $G$.
Clearly any automorphism of $G$ induces an automorphism of $\overline{G}$.
Moreover, the induced map $\Aut G\rightarrow\Aut\overline{G}$ is injective by a standard description of $\Aut G$ in terms of graph automorphisms and inner automorphisms.
Thus we may assume that $G$ is of adjoint type.
Now any semisimple inner automorphism of $G$ is conjugate to $\Int t$ for some $t\in T$.
Let $\bar{W}=N_{\Aut G}(T)/T$.
Then two elements of $T$ are $\Aut G$-conjugate if and only if they are $\bar{W}$-conjugate.

To each element $t\in T$ of order $m$ we associate the homomorphism $\lambda_t:{\mathbb Z}\Phi\rightarrow{\mathbb Z}/m{\mathbb Z}$ such that $\alpha(t)=\zeta^{\lambda_t(\alpha)}$.
If $\mu$ is a Kac diagram of non-twisted type, then by the above discussion the corresponding Kac automorphism of $G$ is $\Int t_\mu$, where $\lambda_{t_\mu}(\alpha_i)=\mu(\alpha_i)$ modulo $m$ for $0\leq i\leq n$.
The action of $\bar{W}$ on $\Phi$ gives rise to an action on the set of homomorphisms ${\mathbb Z}\Phi\rightarrow{\mathbb Z}/m{\mathbb Z}$.
Moreover, this action satisfies $Z_{\bar W}(t)=Z_{\bar W}(\lambda_t)$ for $t\in T$ of order $m$, and thus $Z_{\bar W}(t_\mu)=Z_{\bar W}(\mu)$ for any Kac diagram $\mu$.

It is possible for two distinct Kac diagrams to be conjugate by some element of $W$.
In our circumstances it is easier to describe when Kac automorphisms corresponding to Kac diagrams $\mu_1$ and $\mu_2$ are conjugate by an element of $\bar{W}$: this holds if and only if there is an isomorphism of diagrams $\mu_1\rightarrow\mu_2$
(that is, an isomorphism $\psi$ of graphs which also satisfies $\mu_2(\psi(\alpha))=\mu_1(\alpha)$ for each node $\alpha$).
Pick a set $I$ of representatives for the $\bar{W}$-conjugacy classes of Kac diagrams of order $m$.
Then the number of automorphisms of the form $\Int t$, $t\in T$ which are conjugate to a Kac automorphism of order $m$ is $\sum_{\mu\in I}\#({\bar W})/\#(Z_{\bar W}(\mu))=\sum_{\mu\in I}\#({\bar W})/\#(Z_{\bar W}(t_\mu))$.
If $\charac k=0$, then this sum is equal to the number of elements of $T$ of order $m$.
But therefore it is equal to the number of elements of $T$ of order $m$ in arbitrary characteristic.

Now, suppose $\theta$ is an outer automorphism of $G$.
Since any semisimple automorphism of $G$ fixes some Borel subgroup of $G$ and a maximal torus contained in it \cite[7.5]{steinberg}, any outer automorphism of $G$ is conjugate to an automorphism of the form $\Int t\circ\gamma$, where $t\in T$ and $\gamma$ is one of the graph automorphisms described above.
Let $T(0)=(T^\gamma)^\circ$ and let $T(1)=\{ t\in T\, |\, \gamma(t)=t^{-1}\}^\circ$.
Then $T$ is the almost direct product of $T(0)$ and $T(1)$ by Lemma \ref{stabletori}.
Hence $t=t_0.t_1$ for some $t_0\in T(0)$, $t_1\in T(1)$.
Thus there is some $s\in T(1)$ such that $s^2=t_1$; then $\Int s^{-1}\circ\Int t\circ\gamma\circ\Int s=\Int ts^{-2}\circ\gamma=\Int t_0\circ\gamma$.
Therefore any outer automorphism of $G$ is conjugate to an automorphism of the form $\Int t\circ\gamma$ with $t\in T(0)$.
We claim that $\Int t\circ\gamma$ and $\Int s\circ\gamma$, for $t,s\in T(0)$ are conjugate if and only if there exists some $g\in N_{G^\gamma}(T)$ such that $gtg^{-1}=s$.
Since $G$ is adjoint, $\Int t\circ\gamma$ and $\Int s\circ\gamma$ are conjugate if and only if there exists some $x\in G$ such that $xt\gamma(x^{-1})=s$.
Consider the Bruhat decomposition $x=unv$ of $x$, where $u,v\in U$ and $n\in N_G(T)$.
We have $xt=unt\cdot (t^{-1}vt)$ and $s\gamma(x)=(s\gamma(u)s^{-1})\cdot s\gamma(n)\gamma(v)$, thus by uniqueness $nt=s\gamma(n)$.
In particular, $nT\in W^\gamma$.
Now, a direct check shows that each element of $W^\gamma$ has a representative in $G^\gamma$ (once again using adjointness of $G$), thus $s$ and $t$ are $N_{G^\gamma}(T)$-conjugate.

It follows that the $\Int G$-conjugacy classes of outer automorphisms of the form $\Int x\circ\gamma$ are in one-to-one correspondence with the $W^\gamma$-conjugacy classes in $T$.
Repeating the counting argument used for inner automorphisms above, we deduce that any outer automorphism is conjugate to a Kac automorphism.
\end{proof}

From now on we will refer to the type of an outer automorphism of a simple group $G$ by specifying the type of its corresponding Kac algebra, but without the tilde.
We note the following corollary of the above result for inner automorphisms.
In good characteristic this was already known (\cite[Prop. 30]{mcninchsommers} or \cite[Prop. 3.1]{premorbits}).
Recall that a pseudo-Levi subgroup of $G$ is a subgroup of the form $Z_G(x)^\circ$ for a semisimple element $x$ of $G$.
If $I$ is a proper subset of $\tilde\Delta$ then we denote by $L_I$ the subgroup of $G$ generated by $T$ and all $U_{\pm\alpha_i}$ with $\alpha_i\in I$.

\begin{corollary}\label{pseudocor}
The pseudo-Levi subgroups are the $G$-conjugates of subgroups of the form $L_I$, where $I$ is a proper subset of the $\tilde\Delta$ such that not all elements of $\tilde\Delta\setminus I$ have coefficient divisible by $p$ in the expression for $\nu$.
\end{corollary}

(If $p$ is good then there is no restriction here. The groups $L_I$ which don't satisfy the conditions of Cor. \ref{pseudocor} are the centralizers of non-smooth diagonalizable subgroups \cite{serrekac}.)

For later use we make the following observation.

\begin{lemma}\label{0and1}
Suppose $\mu$ is a Kac diagram corresponding to a positive rank automorphism.
Then $\mu(\alpha)\in\{Ê0,1\}$ for each $\alpha\in\Pi$.
\end{lemma}

\begin{proof}
Suppose $\mu(\alpha)>1$ for some $\alpha\in\Pi$.
Then $\Pi'=\Pi\setminus\{\alpha\}$ is a union of Dynkin diagrams.
If $\mu$ corresponds to an inner automorphism then ${\mathfrak l}={\mathfrak t}\oplus\sum_{\beta\in{\mathbb Z}\Pi'}{\mathfrak g}_\beta$ is the Lie algebra of the subgroup $L_{\Pi'}$ of $G$, in the notation introduced above.
Moreover, the subspace spanned by all ${\mathfrak g}_\beta$, where $\beta$ is non-zero and has non-negative coefficients in the elements of $\Pi'$, is contained in the Lie algebra ${\mathfrak u}_{\Pi'}$ of a maximal unipotent subgroup of $L_{\Pi'}$.
In particular, ${\mathfrak g}(1)\subset{\mathfrak u}_{\Pi'}$ and thus does not contain any semisimple elements.

Suppose therefore that $\mu$ is an outer automorphism of $G$.
If $\mu(\beta_i)>1$ for some $i$ then let $\mu'$ be the Kac diagram with $\mu'(\beta_j)=\delta_{ij}$.
There is a Kac automorphism $\psi$ of $G$ corresponding to $\mu'$.
Moreover, it follows from the description of Kac automorphisms that ${\mathfrak g}(1)\subset{\mathfrak g}^{d\psi}$.
But $\theta|_{G^\psi}$ is a zero-rank (inner) automorphism by the same argument as above.
\end{proof}

\section{Carter's classification of conjugacy classes}

Let $W$ be a Weyl group with natural real representation $V$.
Here we recall the set-up of the classification, according to Carter \cite{carter} of the conjugacy classes in $W$.
Any element $w$ of $W$ can be expressed as a product $w=w_1w_2$, where $w_1^2=w_2^2=1$ and $\{Êv\in V\,|\, w_1\cdot v=-v\}\cap\{Êv\in V\, |\, w_2\cdot v=-v\}=\{Ê0\}$.
Moreover, any involution $w'\in W$ can be expressed as a product of $l$ reflections corresponding to orthogonal roots, where $l=\dim\{ v\in V\, \mid \, w'(v)=-v\}$.
Thus let $I_1$, $I_2$ be subsets of $\Phi$ with $\#(I_i)=l(w_i)$ such that $w_i=\prod_{\alpha\in I_i}s_\alpha$ ($i=1,2$).
Then this gives an expression for $w$ as a product of reflections corresponding to $l_1+l_2=\dim V-\dim V^w$ roots, where $l_i=\dim\{ v\in V\, \mid w_i(v)=-v\}$, $i=1,2$.
Associate a graph $\Gamma$ to $w$ with one node for each root in $I_1\coprod I_2$ and $\langle\alpha,\beta\rangle\langle\beta,\alpha\rangle$ edges between nodes corresponding to distinct roots $\alpha,\beta\in I_1\cup I_2$.
The graph $\Gamma$ constructed in this way is the {\bf admissible diagram} associated to $w$.
Less formally, we will say that $w$ has Carter type $\Gamma$.
(In fact, it is possible to have two different such expressions for $w$ giving rise to two different diagrams, and on the other hand, non-conjugate elements can have the same admissible diagram.
The first problem Carter solved simple by making a particular choice of admissible diagram for each conjugacy class.
For the second problem, which only occurs in types $D$ and $E$, some classes are marked with a single prime and others with a double prime.
There are never more than two classes for the same admissible diagram \cite{carter}.)
For example, if $\Gamma$ is the Dynkin diagram for $W$ then $w$ is a Coxeter element of $W$, while if $\Gamma$ is the trivial graph then $w$ is the identity element.
The irreducible admissible diagrams are classified and their characteristic polynomials determined in \cite[Tables 2-3]{carter}.
We will say that an admissible diagram $\Gamma$ has order $N$ if an element of the corresponding conjugacy class has order $N$.
Let $w$ be an element of $W$ with admissible diagram $\Gamma$.
We denote by $\Phi_1$ the (closed) subsystem of $\Phi$ spanned by the elements of $\Gamma$, and by $\Phi_2$ the set of roots $\alpha\in\Phi$ which are orthogonal to all roots in $\Phi_1$.
%Let $\overline{\Phi}_1$ be the intersection ${\mathbb Q}\Phi_1\cap\Phi$, a subsystem of $\Phi$ which contains $\Phi_1$.
%(This was denoted $\bar{W}_1$ by Carter.)
In \cite[Tables 8-12]{carter}, Carter lists the conjugacy classes in the exceptional type Weyl groups, classified by admissible diagrams, and gives the number of elements in each class (and thus the number of elements in the centralizer of an element of a given class).
Following Carter's notation, we will use a tilde to denote a subsystem consisting of short roots.

\begin{example}
We give some examples to illustrate the above notions.

a) Let $W=S_n$, the symmetric group on $n$ letters.
Then the admissible diagram correponding to an $m$-cycle is a Coxeter graph of type $A_{m-1}$.
More generally, any element $w$ of $W$ is a product of disjoint cycles, of length $m_1,\ldots ,m_r$ say.
Then the admissible diagram for $w$ is a disjoint union of Coxeter graphs $A_{m_1-1},\ldots ,A_{m_r-1}$.
We say that $w$ is of type $A_{m_1-1}\times\ldots\times A_{m_r-1}$.
In this case $\Phi_1$ is isomorphic to the union $A_{m_1-1}\cup\ldots\cup A_{m_r-1}$ and $\Phi_2$ is isomorphic to $A_{n-m_1-\ldots-m_r-1}$.

b) If $W$ is a Weyl group of type $B_n$ then any element $w\in W$ can be decomposed as a product of {\it signed cycles}: a positive $m$-cycle has order $m$, and a negative $m$-cycle has order $2m$.
The admissible diagram for a positive $m$-cycle is a Coxeter graph of type $A_{m-1}$; the admissible diagram for a negative $m$-cycle is a Coxeter graph of type $B_m$.
If $w$ is a positive $m$-cycle then $\Phi_2$ is isomorphic to $B_{n-m}$ unless $m=2$, in which case $\Phi_2$ is isomorphic to $B_{n-2}\cup A_1$.
If $w$ is a negative $m$-cycle then $\Phi_2$ is isomorphic to $D_{n+1-m}$.

c) If $W$ is a Weyl group of type $D_n$ then $W$ is contained in a Weyl group of type $B_n$.
Thus any element of $W$ can also be expressed as a product of disjoint signed cycles.
(On the other hand, two elements of $W$ may have the same signed cycle type but not be conjugate in $W$.)
A product $w$ of two negative 2-cycles has admissible diagram labelled $D_4(a_1)$.
\end{example}

We want to ensure that all of the relevant information from \cite{carter} can also be applied in positive characteristic, including the information about the characteristic polynomials of elements of $W$.
Consider the `natural' representation of $W$ to be its representation in a Cartan subalgebra of the corresponding complex simple Lie algebra.
For $w\in W$, denote by $c_w(t)$ the characteristic polynomial of $w$ in this natural representation.

\begin{lemma}
Let $\charac k=p>0$ and let $G$ be a simple group over $k$ with root system $\Phi=\Phi(G,T)$, where $T$ is a maximal torus of $G$.
Then the minimal polynomial of $w$ as an automorphism of ${\mathfrak t}=\Lie(T)$ is just the reduction modulo $p$ of $c_w(t)$.
\end{lemma}

\begin{proof}
Given a basis $\{ \chi_1,\ldots ,\chi_n\}$ of the lattice of cocharacters $Y(T)$ we can associate a matrix $A\in\GL(n,{\mathbb Z})$ to $w$ by: $w(\chi_i)=\sum_j A_{ji}\chi_j$.
Since $w(d\chi_i(1))=\sum_j A_{ji}d\chi_j(1)$ and $\{ d\chi_1(1),\ldots ,d\chi_n(1)\}$ is a basis of ${\mathfrak t}$, clearly the characteristic polynomial of $w$ as an element of $\GL({\mathfrak t})$ is just the reduction modulo $p$ of the characteristic polynomial of $A$.
On the other hand, the characteristic polynomial of $A$ is invariant under change of basis in $Y(T)\otimes_{\mathbb Z}{\mathbb Q}$, and in particular is the same if one replaces $G$ by its universal covering.
But now one can choose a basis for $Y(T)$ consisting of the $\alpha^\vee$ where $\alpha$ is a simple root.
Since the matrix thus associated to $w$ is clearly independent of the characteristic, its characteristic polynomial must be equal to $c_w(t)$.
\end{proof}

\begin{lemma}\label{w0prep}
Let $G,T,{\mathfrak t},\Phi,W,w,\Phi_1,\Phi_2$ be as above and suppose that $w$ has order $m$, $p\nmid m$.
Let $U=\{ t\in{\mathfrak t}\,|\,w(t)=\zeta t\}$, where $\zeta$ is a primitive $m$-th root of unity in $k$.
Then any element of $W(\Phi_2)$ acts trivially on $U$.
\end{lemma}

\begin{proof}
Since $G$ is simply-connected, ${\mathfrak t}$ is spanned by elements $h_\alpha$, with $\alpha$ in a basis of simple roots.
Let ${\mathfrak t}_{\Phi_1}$ be the linear span of the $h_\alpha$ with $\alpha\in\Phi_1$ and let $V=\{ t\in{\mathfrak t}\,|\, d\alpha(t)=0\;\forall\alpha\in \Phi_1\}$.
Then ${\mathfrak t}={\mathfrak t}_{\Phi_1}\oplus V$ by consideration of dimensions, and clearly $U\subset{\mathfrak t}_{\Phi_1}$.
But if $\beta\in\Phi_2$ then $s_\beta(h_\alpha)=h_\alpha$ for all $\alpha\in\Phi_1$, thus any element of $W(\Phi_2)$ acts trivially on $U$.
\end{proof}

We remark that a suitable generalization of Lemma \ref{w0prep} holds for any orthogonal subsystems $\Phi_1$, $\Phi_2$ and an element $w\in W(\Phi_1)$.

\section{Determination of positive rank automorphisms}

In this section we give details of the calculations we use to determine the positive rank automorphisms.
Assume that $G$ is simple, simply-connected and of exceptional type, and that $\charac k$ is either zero or a good prime for $G$.
(If $G$ is simple and of exceptional type, then the assumption that $p$ is good implies that $G$ is separably isogenous to a group satisfying the standard hypotheses.)
In particular we assume that $p>3$, and therefore that all elements of a Weyl group of type $G_2$ or $F_4$ are semisimple.
(We recall that $W(G_2)$ is a dihedral group of order $12$, and $W(F_4)$ has order $1152=2^7.3^2$.)
Let ${\mathfrak c}$ be a Cartan subspace of ${\mathfrak g}(1)$ and let $T_1$ be the (unique) maximal $\theta$-split torus of $G$ such that ${\mathfrak c}\subset\Lie(T_1)$ (Lemma \ref{stabtori}(b)).
Recall that $Z_G({\mathfrak c})^\circ= Z_G(T_1)$ \cite[Rk. 2.8(c)]{thetagroups}; since $G$ is simply-connected, $Z_G({\mathfrak c})$ is connected and hence equals $Z_G(T_1)$.
Let $T(0)$ be a maximal torus of $Z_{G(0)}({\mathfrak c})^\circ$.
Then $Z_G(T(0))\cap Z_G({\mathfrak c})$ is a $\theta$-stable maximal torus of $G$ \cite[Lemma 4.1]{thetagroups}.
In this section we will fix ${\mathfrak c}$ and $T(0)$, and set $T=Z_G(T(0))\cap Z_G({\mathfrak c})$, ${\mathfrak t}=\Lie(T)$.
With this assumption on $T$, we make some straightforward observations.

\begin{lemma}
Suppose $\theta$ is inner.
Then $\theta=\Int n_w$, where $n_w\in N_G(T)$.
Moreover, either $w=n_w T$ has order $m$, or $n_w\in T$ and $\theta$ is of zero rank.
\end{lemma}

\begin{proof}
Since $T$ is $\theta$-stable, the first statement is obvious.
But if $n_w$ has order less than $m$, then $\{Êt\in {\mathfrak t}\, |\,d\theta(t)=\zeta t\}$ is trivial, and therefore by our choice of $T$, $T(0)=T$ and $n_w\in T$.
\end{proof}

In fact, we can make a more precise statement than this.

\begin{lemma}\label{orders}
Either $n_w\in T$ or $w=n_w T$ has Carter type $\Gamma^{(1)}\times\ldots\times\Gamma^{(l)}$, where the $\Gamma^{(i)}$, $1\leq i\leq l$ are disjoint irreducible diagrams and at least one has order $m$.
\end{lemma}

\begin{proof}
Let $\Phi^{(i)}$ be the smallest root subsystem of $\Phi$ containing all the elements of $\Gamma^{(i)}$, let ${\mathfrak t}^{(i)}$ be the vector subspace of ${\mathfrak t}$ spanned by all $h_\alpha$ with $\alpha\in\Phi^{(i)}$ and let $w=w^{(1)}\ldots w^{(l)}$, where $w^{(i)}$ is a product of reflections corresponding to the vertices of $\Gamma^{(i)}$.
(Here we define $h_\alpha=d\alpha^\vee(1)$.)
Since the roots in $\Gamma^{(1)}\cup\ldots\cup\Gamma^{(l)}$ are linearly independent (\cite[\S 4]{carter}), and since $G$ is simply-connected, it follows that ${\mathfrak t}$ contains the direct sum ${\mathfrak t}^{(1)}\oplus\ldots\oplus{\mathfrak t}^{(l)}$.
Clearly $w$ preserves each of the subspaces $\Phi^{(i)}$ and acts on $\Phi^{(i)}$ by $w^{(i)}$.
Hence $\{ t\in{\mathfrak t}\, |\, w(t)=\zeta t\}$ is spanned by the subspaces $\{ t\in{\mathfrak t}^{(i)}\, |\, w^{(i)}(t)=\zeta t\}$.
In particular, if none of the $w^{(i)}$ has order $m$ then ${\mathfrak c}$ is trivial, hence so is $w$ by our assumption on $T(0)$.
\end{proof}

Lemma \ref{orders} excludes certain possibilities for $w$, such as an element of type $A_1\times A_2$ (since here $m=6$, but no irreducible subdiagram has order 6).
Later we will see that in fact each of the $\Gamma^{(i)}$ must have order $m$.
(This was already established for the classical case in \cite{thetagroups}.
For outer automorphisms the situation is more complicated.)

\begin{lemma}\label{cox}
Let $G$ be simple and simply-connected group and let $h$ be the Coxeter number of $G$.
Assume the characteristic of the ground field is coprime to $h$.
If $n_w\in N_G(T)$ is such that $w=n_w T$ is a Coxeter element in $W$, then $\Int n_w$ is an automorphism of $G$ of order $h$ and the corresponding Kac diagram is the diagram with 1 on every node.
\end{lemma}

\begin{proof}
For the first part we just have to show that $n_w^h\in Z(G)$.
Clearly $t=n_w^h$ satisfies $w(t)=t$.
We claim that this implies $\alpha_i(t)=1$ for each $i$, hence $t=Z(G)$.
To prove the claim, we may assume that $w=s_{\alpha_r}\ldots s_{\alpha_1}$ where $\{\alpha_1,\ldots,\alpha_r\}$ is a basis of simple roots of $\Phi(G,T)$.
Let $t_0=t$, $t_i=s_{\alpha_i}\ldots s_{\alpha_1}(t)$ for $1\leq i\leq r$.
Then $t_{i+1}=t_i\cdot \alpha_{i+1}^\vee(\alpha_i(t_i))$.
Hence $t=t\cdot \alpha_r^\vee(\alpha_r(t_{r-1}))\ldots \alpha_1^\vee(\alpha_1(t))$.
Since $G$ is simply-connected, this implies that $\alpha_r(t_{r-1})=\ldots=\alpha_1(t)=1$.
But then $t=t_1=\ldots=t_{r-1}$, thus $\alpha_i(t)=1$ for $1\leq i\leq r$.
Therefore $t\in Z(G)$.
Now, since any two maximal tori of $G$ are conjugate, any two Coxeter elements in $W$ are conjugate, and any two representatives of a Coxeter element $w$ are $T$-conjugate by Lemma \ref{fpfree}, it follows that there is a unique conjugacy class ${\cal C}$ of automorphisms of $G$ which act as a Coxeter element on some maximal torus.
We claim that the Kac automorphism corresponding to the Kac diagram with 1 on every node is of positive rank.
Indeed, it follows from the fact that $p\nmid h$ that $G$ satisfies the standard hypotheses.
Thus we can construct a KW-section for the adjoint representation as follows: let $e=\sum_{\alpha\in\Delta}e_\alpha$ be a regular nilpotent element, let $\lambda:k^\times\rightarrow T$ be the unique cocharacter such that $\alpha(\lambda(t))=t^2$ for $\alpha\in\Delta$ and let ${\mathfrak m}$ be an $\Ad \lambda(k^\times)$-stable subalgebra of ${\mathfrak u}_-$ such that ${\mathfrak g}={\mathfrak u}_-\oplus[e,{\mathfrak g}]$.
(Here ${\mathfrak u}_-=\sum_{\alpha\in -\Phi^+}{\mathfrak g}_\alpha$.)
Then $e+{\mathfrak m}$ is a KW-section and hence $e+m$ is non-nilpotent for any non-zero $m\in{\mathfrak m}$.
Now, since ${\mathfrak m}$ is $\Ad\lambda(k^\times)$-stable, we must have $e_{-\hat\alpha}\in{\mathfrak m}$, and therefore $c=e+e_{-\hat\alpha}\in{\mathfrak g}(1)\cap(e+{\mathfrak m})$.
Since $c$ is not nilpotent, we deduce that the Kac automorphism with all coefficients equal to 1 is positive rank.
Since it is the only Kac automorphism of order $h$ with coefficients $1$ or $0$, the lemma follows.
\end{proof}

We will say that an automorphism which is conjugate to $\Int n_w$, where $n_w\in N_G(T)$ represents a Coxeter element in $W$, is a {\it Coxeter automorphism}.
By Lemma \ref{cox}, the Coxeter automorphisms form a single conjugacy class in $\Int G$.
The following lemma is presumably known, but we could not find a clear reference in the literature.

\begin{lemma}\label{cox2}
Let $n_w\in N_G(T)$ be such that $n_w T$ represents a Coxeter element in $W$.
Assume $p$ is coprime to $h$ as before.
Then $n_w^h$ is as follows:

a) $G=\SL(n,k)$, then $h=n$ and $n_w^h=I$ if $n$ is odd, $n_w^h=-I$ if $n$ is even.

b) $G=\Spin(2n+1,k)$, then $h=2n$ and $n_w^h=1$ if $n\equiv 0$ or $3$ modulo 4, while $n_w^h=\alpha_n^\vee(-1)$ otherwise.

c) $G=\Sp(2n,k)$, then $h=2n$ and $n_w^h=-I$.

d) $G=\Spin(2n,k)$ then $h=2(n-1)$ and $n_w^h=1$ if $n\equiv 0$ or $1$ modulo 4, while $n_w^h=\alpha_{n-1}^\vee(-1)\alpha_n^\vee(-1)$ otherwise.

e) $G$ of type $G_2$ (resp. $F_4$, $E_6$, $E_8$) then $h=6$ (resp. $h=12$, $h=12$, $h=30$) and $n_w^h=1$.

g) $G$ of type $E_7$, then $h=18$ and $n_w^h=\alpha_2^\vee(-1)\alpha_5^\vee(-1)\alpha_7^\vee(-1)$.
\end{lemma}

\begin{proof}
By Lemma \ref{cox}, $n_w^h\in Z(G)$ in each case.
Let $\zeta$ be a primitive $h$-th root of unity.
Types $A$ or $C$ can be calculated directly.
If $G$ is of type $G_2$, $F_4$ or $E_8$ then $Z(G)$ is trivial.
If $G$ is of type $B_n$, then we use Lemma \ref{cox}: $\Int n_w$ is conjugate to the Kac automorphism which sends each $e_\alpha$, $\alpha\in\Delta$ to $\zeta e_\alpha$, where $\zeta$ is a primitive $h$-th root of unity.
But this automorphism is just $\Int t$, where $t=\prod_{i=1}^{n-1}\alpha_i^\vee(\zeta^{in-i(i-1)/2})\cdot \alpha_n^\vee(\xi^{n(n+1)/2})$, $\xi$ a square-root of $\zeta$.
Thus $t^h$ is as described.
The calculation for type $D_n$ is similar: $\Int n_w$ is conjugate to $\Int t$, where $t=\prod_{i=0}^{n-2}\alpha_i^\vee(\zeta^{i(n-1)-i(i-1)/2})\cdot\alpha_{n-1}^\vee(\xi^{n(n-1)/2})\alpha_n^\vee(\xi^{n(n-1)/2})$ and thus $t^h=1$ if and only if $n\equiv 0$ or $1$ modulo 4.
If $G$ is of type $E_6$ then let $\gamma$ be an outer automorphism of $G$.
Then $\gamma$ permutes the two non-identity elements of $Z(G)$.
On the other hand, $\Int \gamma(n_w)\in{\cal C}$ and thus $\gamma(n_w^{12})=n_w^{12}$, whence $n_w^{12}=1$.
Finally, if $G$ is of type $E_7$ then $w$ is of order 18 and $\Int n_w$ is conjugate to $\Int t$, where $t=\alpha_1^\vee(\xi^{34})\alpha_2^\vee(\xi^{49})\alpha_3^\vee(\xi^{66})\alpha_4^\vee(\xi^{96})\alpha_5^\vee(\xi^{75})\alpha_6^\vee(\xi^{52})\alpha_7^\vee(\xi^{27})$, and thus $n_w^{18}=t^{18}=\alpha_2^\vee(-1)\alpha_5^\vee(-1)\alpha_7^\vee(-1)$.
\end{proof}

Recall that an automorphism $\theta$ is N-regular if ${\mathfrak g}(1)$ contains a regular nilpotent element of ${\mathfrak g}$, and that an automorphism $\theta$ of order $m$ is of {\bf maximal rank} if $\dim{\mathfrak c}\cdot\varphi(m)=\rank G$, where $\varphi$ denotes the Euler number of $m$.

\begin{lemma}\label{coxpowers}
Let $w$ be one of the Weyl group elements on the following list.
Then all representatives $n_w$ of $w$ in $N_G(T)$ are $T$-conjugate and, modulo the centre of $G$, have the same order as $w$.
The (unique) Kac diagram corresponding to $\Int n_w$, the order of $w$ (and thus of $\Int n_w$) and the rank of $\Int n_w$ are as given below.

a) $G$ of type $G_2$: $w$ of type $G_2$, $A_2$ or $A_1\times \tilde{A_1}$.
The respective Kac diagrams are: $111$, $011$ and $010$, the respective orders of $w$ are 6, 3 and 2 and the respective ranks of $\Int n_w$ are 1, 1 and 2.

b) $G$ of type $F_4$: $w$ of type $F_4$, $F_4(a_1)$, $D_4(a_1)$, $A_2\times\tilde{A_2}$ and $A_1^4$.
The corresponding Kac diagrams are: $11111$, $10101$, $10100$, $00100$ and $01000$, the respective orders of $w$ are 12, 6, 4, 3 and 2 and the respective ranks are 1, 1, 2, 2, and 4.

Each such automorphism is N-regular.
\end{lemma}

\begin{proof}
Each $w$ listed in the Lemma is of maximal rank and is a power of the Coxeter element by inspection of \cite[Table 3]{carter}.
(In particular, the orders and ranks are as given in the lemma.)
Thus all representatives of $w$ are $T$-conjugate by Lemma \ref{fpfree}.
It remains to check that the Kac diagrams are as stated.
For the Coxeter elements this is Lemma \ref{cox2}.
Suppose on the other hand that $l$ divides the Coxeter number.
Then the $l$-th power of a Coxeter automorphism is conjugate to the automorphism which sends $e_{\pm\alpha}$ ($\alpha$ simple) to $\zeta^{\pm l} e_{\pm\alpha}$, where $\zeta$ is a primitive $h$-th root of unity ($h$ the Coxeter number).
(We note that since $p$ is good it doesn't divide $h$ in type $G_2$ or $F_4$.)
In particular, ${\mathfrak g}(0)$ is conjugate to the subalgebra spanned by ${\mathfrak t}$ and all root subspaces ${\mathfrak g}_\alpha$ where $\alpha$ has length divisible by $l$.
Thus, inspecting the possibilities for automorphisms of order $h/l$, it is straightforward to determine the Kac diagram in each case.

Type $A_2$ in $G_2$: this is the square of the Coxeter element, and hence ${\mathfrak g}(0)$ is conjugate to ${\mathfrak t}\oplus{\mathfrak g}_{\pm{(2\alpha_1+\alpha_2)}}$.

Type $A_1\times\tilde{A_1}$ in $G_2$: this is the cube of a Coxeter element, and hence ${\mathfrak g}(0)$ is conjugate to ${\mathfrak t}\oplus{\mathfrak g}_{\pm(\alpha_1+\alpha_2)}\oplus{\mathfrak g}_{\pm(3\alpha_1+\alpha_2)}$.

Type $F_4(a_1)$ in $F_4$: this is the square of a Coxeter element and hence ${\mathfrak g}(0)$ is conjugate to ${\mathfrak t}\oplus{\mathfrak g}_{\pm 1221}\oplus{\mathfrak g}_{\pm 1122}$.

Type $D_4(a_1)$ in $F_4$: this is the cube of a Coxeter element and hence after conjugation ${\mathfrak g}(0)={\mathfrak t}\oplus{\mathfrak g}_{\pm 1111}\oplus{\mathfrak g}_{\pm 1120}\oplus{\mathfrak g}_{\pm 0121}\oplus{\mathfrak g}_{\pm 1232}$.
Here $1111$ and $0121$ generate a root system of type $\tilde{A_2}$, while $1120$ is a long root, thus ${\mathfrak g}(0)$ is of type $\tilde{A_2}\times A_1$.

Type $A_2\times \tilde{A_2}$ in $F_4$: this is a Coxeter element raised to the fourth power, hence ${\mathfrak g}(0)$ is conjugate to ${\mathfrak t}\oplus{\mathfrak g}_{\pm 1110}\oplus{\mathfrak g}_{\pm 0120}\oplus{\mathfrak g}_{\pm 0111}\oplus{\mathfrak g}_{\pm 1221}\oplus{\mathfrak g}_{\pm 1122}\oplus{\mathfrak g}_{\pm 1242}$.
Here $1110$, $0111$ generate a root subsystem of $\Phi$ of type $\tilde{A_2}$ while $0120$, $1122$ generate a root system of type ${A_2}$.

Type $A_1^4$ in $F_4$: this is the sixth power of a Coxeter element, and hence ${\mathfrak g}(0)$ is conjugate to the subalgebra spanned by ${\mathfrak t}$ and all root spaces ${\mathfrak g}_\alpha$ where $\alpha$ is of even length.
But the roots of even length have basis: 1100, 0011, 0110 and 1120, where the first three span a subsystem of type $C_3$ and 1120 is a long root.

It is now easily seen that the only possible Kac diagrams with coefficients 0 and 1 which could correspond to these automorphisms are those stated in the lemma.
\end{proof}

In type $G_2$ the only remaining non-trivial conjugacy classes in $W$ are involutions.
Moreover, it is easy to see from Lemma \ref{kacgen} (or see, for example \cite{springerinvs}) that there is a unique class of involution of a simple group of type $G_2$ and hence Lemma \ref{coxpowers} gives us a complete list of positive rank automorphisms.
(See Table 1.)

In type $F_4$, we are reduced to studying elements of the Weyl group of the following types: $A_2$, $\tilde{A}_2$ (order 3), $B_2$, $B_2\times A_1$, $A_3$, $A_3\times\tilde{A_1}$ (order 4), $C_3$, $B_3$, $C_3\times A_1$, $D_4$ (order 6), $B_4$ (order 8).

Consider an admissible diagram $\Gamma=\Gamma'\cup\Gamma''$, where $\Gamma'$ is a union of irreducible subdiagrams of order $m$, and $\Gamma''$ is a union of irreducible subdiagrams of order less than $m$.
As before let $\Phi_1$ be the smallest root subsystem of $\Phi$ containing all of the roots in $\Gamma'$, and let $\Phi_2$ be the set of roots in $\Phi$ which are orthogonal to all elements of $\Phi_1$.
Let $L_i$ ($i=1,2$) be the subgroup of $G$ generated by all root subgroups $U_\alpha$, $\alpha\in\Phi_i$ and let $S$ be the torus generated by all $\alpha^\vee(k^\times)$, $\alpha\in\Phi_1$.
Clearly $S$ is a maximal torus of $L_1$ and $Z_G(S)'=L_2$.

Let $w=w'w''$ be an element in the conjugacy class corresponding to $\Gamma$, where $w'$ corresponds to $\Gamma'$ and $w''$ corresponds to $\Gamma''$.
Suppose $\theta=\Int n_w$, where $n_w T=w\in W$.
Since $T$ and $S$ are $\theta$-stable, it follows that $L_1$ and $L_2$ are $\theta$-stable.
Moreover, $w|_S=w'|_S$ and $(S^w)^\circ$ is trivial, thus there exists $n_{w'}\in N_{L_1}(S)$ such that $\theta|_{S}=\Int n_{w'}|_{S}$.
Since $(\Int n_{w'}^{-1}\circ\theta)|_{L_1}$ acts trivially on $S$, we can replace $n_{w'}$ by an element of the form $n_{w'}s$, $s\in S$ such that $\theta|_{L_1}=\Int n_{w'}|_{L_1}$.
Therefore $n_w=n_{w'}g$ for some $g\in Z_G(L_1)$.
Recall that we assume that $T(0)$ is maximal in $Z_G({\mathfrak c})^\theta$.

\begin{lemma}\label{orthog}
$w''$ is trivial.
\end{lemma}

\begin{proof}
We have $g\in Z_G(L_1)=Z'\cdot Z_{L_2}(L_1)$ where $Z'=\{ t\in T\, |\, \alpha(t)=1\,\,\mbox{for all}\,\,\alpha\in\Phi_1\}$.
Note that $(Z')^\circ\cdot S=T$, $w'$ acts trivially on $Z'$ and has finitely many fixed points on $S$.
Hence $S'=(Z')^\circ=(T^{w'})^\circ$ and $H=L_2 S'=S'L_2$.
Then $Z'=S'Z(L_1)$ and thus, after replacing $n_{w'}$ by an appropriate element of $n_{w'}Z(L_1)$ if necessary, we may assume that $g\in Z_H(L_1)=Z_{L_2}(L_1)S'$.
Let $h\in L_1$ be such that $s=hn_{w'}h^{-1}\in S$.
Then $hn_wh^{-1}=shgh^{-1}=sg$.
Moreover, $s$ and $g$ commute and $(sg)^m=1$, thus $g$ is semisimple.
We deduce that there exists $h'\in L_2$ such that $h'gh'^{-1}=s'\in S'$.
Thus $n_w$ is $G$-conjugate to $ss'\in T$.
But therefore $n_w$ is conjugate to $h^{-1}ss'h=n_{w'}s'$.
By our assumption of the maximality of $T(0)$, we must have $w=w'$ in the first place.
\end{proof}

This result eliminates the cases $A_3\times \tilde{A_1}$, $C_3\times A_1$ and $B_2\times A_1$.
But it also eliminates the conjugacy class of type $D_4$ by the same argument, since $D_4$ has an alternative admissible diagram of type $B_3\times \tilde{A_1}$.
(This can easily be seen in the Weyl group of type $B_4$, in which an element of type $D_4$ is a product of a negative 3-cycle and a negative 1-cycle.)

To continue, we make a few observations in a similar vein to Lemma \ref{coxpowers}.

\begin{lemma}
There exists a unique conjugacy class of positive rank automorphism of order 8 in type $F_4$.
The corresponding Kac diagram is $11101$.
\end{lemma}

\begin{proof}
Let $w$ be an element of $W$ of type $B_4$.
We observe that $B_4$ is the unique class in $W$ of order 8 \cite[Tables 3,8]{carter}.
Moreover, any two representatives $n_w$, $n_wt$ of $w$ in $N_G(T)$ are $T$-conjugate (Lemma \ref{fpfree}) and $n_w^4=(n_wt)^4=1$ (Lemma \ref{cox2}), thus there is a unique conjugacy class of positive rank automorphisms of order 8.
It remains to check that the Kac diagram is as indicated in the Lemma.
Consider the $B_4$-type subsystem of $\Phi$ with basis $\{Ê-\hat\alpha=\alpha_0,\alpha_1,\alpha_2,\alpha_3\}$.
Then $\Int n_w$ is conjugate to the automorphism of $G$ which sends $e_{\alpha_i}$ to $\zeta e_{\alpha_i}$ for $i=0,1,2,3$, where $\zeta$ is a primitive 8-th root of unity.
But for such an automorphism either $e_{\alpha_4}\mapsto\zeta^{-1} e_{\alpha_4}$ and thus ${\mathfrak g}(0)={\mathfrak t}\oplus{\mathfrak g}_{\pm(\alpha_3+\alpha_4)}$, or $e_{\alpha_4}\mapsto\zeta^3 e_{\alpha_4}$, in which case ${\mathfrak g}(0)={\mathfrak t}\oplus{\mathfrak g}_{\pm 1221}$.
Since $11101$ is the only Kac diagram of order 8 which has coefficients 1 and 0 and has fixed point subalgebra of type $\tilde{A_1}$, this proves the Lemma.
\end{proof}

\begin{lemma}\label{order6}
There are two conjugacy classes of automorphisms of $F_4$ of order 6 and rank 1.
They have Kac diagrams $01010$ and $11100$ and each element of the first (resp. second) class acts on some maximal torus as a Weyl group element of type $C_3$ (resp. $B_3$).
\end{lemma}

\begin{proof}
We note first of all that there are 4 classes of automorphisms of order 6 with Kac coefficients 0 and 1: the Kac diagrams $10101$, $11100$, $01010$ and $00011$.
In addition, we have seen above that $10101$ is the unique class of automorphism of order 6 and rank two.
Thus any automorphism of rank 6 and order 1 is conjugate to $\Int n_w$ where $w=n_w T\in W$ is either of type $C_3$ or of type $B_3$.

Let $\theta=\Int t$ be a Kac automorphism of type $01010$ and consider the root subsystem $\Sigma$ of $\Phi$ generated by $\alpha_3+\alpha_4$, $\alpha_2+\alpha_3$ and $\alpha_1$.
Then $\Sigma$ is of type $C_3$ and if we let $H$ be the (Levi) subgroup of $G$ generated by $T$ and the subgroups $U_{\pm\alpha}$ with $\alpha\in\Sigma$ then $\theta$ stabilizes $H$ and is easily seen to be a Coxeter automorphism of $H'$ (since $d\theta(e_\alpha)=\zeta e_\alpha$ for $\alpha=\alpha_3+\alpha_4,\alpha_2+\alpha_3,\alpha_1$).
Thus $t$ is $H$-conjugate to some $n_w\in N_G(T)$, where $w=n_w T$ is of type $C_3$.

Let $\theta=\Int t$ be the Kac automorphism for the Kac diagram $11100$ and consider the root subsystem $\Sigma$ of $\Phi$ generated by $\alpha_1$, $\alpha_2$ and $\alpha_3+\alpha_4$.
Then $\Sigma$ is of type $B_3$ and $d\theta(e_\alpha)=\zeta e_\alpha$ for $\alpha=\alpha_1,\alpha_2,\alpha_3+\alpha_4$.
Thus if we let $H$ be the Levi subgroup of $G$ generated by $T$ and all $\alpha$ with $\alpha\in\Sigma$ then $\theta|_{H'}$ is a Coxeter automorphism.
In particular, $t$ is $H$-conjugate to some $n_w\in N_G(T)$, where $w=n_w T$ is of type $B_3$.

It remains to show that a Kac automorphism of type $00011$ is of zero rank.
Here $G(0)$ is the pseudo-Levi subgroup generated by $T$ and $U_\alpha$ with $\alpha=\pm\alpha_2,\pm\alpha_1,\pm\hat\alpha$.
In particular, $\dim G(0)=16$.
We have $\dim{\mathfrak g}(1)=5$; thus it will suffice to show that there exists $e\in{\mathfrak g}(1)$ such that $\dim Z_{G(0)}(e)=11$.
Let $e=e_{\alpha_3}+e_{\alpha_4}$.
Then ${\mathfrak z}_{{\mathfrak g}(0)}(e)$ is spanned by ${\mathfrak g}_\alpha$ with $\alpha=\pm\hat\alpha,\pm\alpha_1,\pm 1342,-\alpha_2,-(\alpha_1+\alpha_2),1242$ and a two-dimensional subalgebra of ${\mathfrak t}$.
In particular, $\dim Z_{G(0)}(e)\leq 11$, whence $\dim G(0)\cdot e\geq 5$.
It follows that $G(0)\cdot e$ is dense in ${\mathfrak g}(1)$ and therefore $\theta$ is of zero rank.
\end{proof}

\begin{corollary}\label{order3}
There are two conjugacy classes of automorphisms of $F_4$ of order 3 and rank 1.
They have Kac diagrams $11000$ and $10001$ and each element of the first (resp. second) class acts on some maximal torus of $G$ as a Weyl group element of type $A_2$ (resp. $\tilde{A}_2$).
\end{corollary}

\begin{proof}
Apart from the (unique) rank 2 automorphism $00100$, the two Kac automorphisms given in the corollary are the only remaining Kac automorphisms of order 3.
Let $\theta$ be the square of the Kac automorphism with diagram $01010$: then ${\mathfrak g}(0)$ is isomorphic to $\mathfrak{so}(7,k)\oplus k$, with basis of simple roots $\{ \alpha_2,1120,\alpha_4\}$.
Thus $\theta$ is conjugate to a Kac automorphism with diagram $10001$.
The statement in this case now follows since $\theta$ must therefore have rank 1, and the square of an element of type $C_3$ is an element of type $\tilde{A_2}$.
Now let $\theta$ be the square of the Kac automorphism with diagram $11100$.
Then ${\mathfrak g}(0)$ is isomorphic to $\mathfrak{sp}(6)\oplus k$, with basis of simple roots $\{Ê\alpha_4,\alpha_3,1220\}$.
Thus, by the same reasoning, the Kac automorphism with diagram $11000$ has rank 1 and corresponds to a Weyl group element of type $A_2$.
\end{proof}

\begin{lemma}
There is one conjugacy class of automorphisms in type $F_4$ of order 4 and rank 1.
The corresponding Kac diagram is $01001$ and is represented by an element of $W$ of type $B_2$.
\end{lemma}

\begin{proof}
By Lemma \ref{0and1} the Kac diagram corresponding to a positive rank automorphism can only have 0s and 1s on the nodes.
Moreover, we have already seen in Lemma \ref{coxpowers}(b) that the Kac automorphism with diagram $10100$ is the unique class of automorphisms of order 4 and rank 2.
Hence there are only two other possibilities: $01001$ and $00010$.
It is easy to see that the Kac automorphism with diagram $01001$ is of positive rank since $e_{1100}+e_{0011}+f_{1122}$ is semisimple.
Hence it is of rank 1.
(By considering the root subsystem generated by $1100$ and $0011$, we can see that $\theta$ is conjugate to an automorphism of the form $\Int n_w$, where $w=n_wT$ is an element of $W$ of type $B_2$.)
It remains to prove that the Kac automorphism with diagram $00010$ is of zero rank (first proved by Vinberg \cite[\S 9]{vin}).
In this case ${\mathfrak g}(0)\cong \mathfrak{sl}(4,k)\oplus\mathfrak{sl}(2,k)$ and ${\mathfrak g}(1)$, as a ${\mathfrak g}(0)$-module, is isomorphic to the tensor product of the natural representation for $\mathfrak{sl}(4)$ and the natural representation for $\mathfrak{sl}(2)$.
In particular, $\dim{\mathfrak g}(1)=8$.
If we let $e=e_{0010}+f_{1231}$ then it is easy to check that ${\mathfrak z}_{{\mathfrak g}(0)}(e)$ is spanned by $h_{\alpha_1}$, $e_{\alpha_1}$, $f_{\alpha_1}$, $h_{\alpha_2}+h_{\alpha_3}+h_{\alpha_4}$, $f_{\hat\alpha}$, $f_{1342}$, $f_{1242}+f_{\alpha_4}$, $e_{1242}+e_{\alpha_4}$, $f_{1100}$ and $f_{\alpha_2}$.
Thus $\dim Z_{G(0)}(e)\leq 10$ and therefore $\dim G(0)\cdot e\geq 8$.
Hence $G(0)\cdot e$ is dense in ${\mathfrak g}(1)$, and so ${\mathfrak g}(1)$ is of zero rank.
\end{proof}

This completes the determination of the positive rank automorphisms in type $F_4$.
We will use these results on automorphisms in type $F_4$ to study the positive rank triality automorphisms in type $D_4$.
In \cite{vin} all automorphisms of $\mathfrak{so}(2n,{\mathbb C})$ of the form $x\mapsto gxg^{-1}$, $g\in{\rm O}(2n,k)$ were studied; this was generalized to (odd) positive characteristic in \cite{thetagroups}, and it was proved that all such $\theta$-groups have a KW-section.
If $n\geq 5$ then all automorphisms have this form; in type $D_4$ this leaves only the automorphisms of type $D_4^{(3)}$.
Recall that the set of long roots in a root system of type $F_4$ is a root system of type $D_4$, and correspondingly $W(F_4)$ contains $W(D_4)$ as a normal subgroup of index 6.
Let $\hat{G}$ be a simple algebraic group of type $F_4$, let $T$ be a maximal torus of $\hat{G}$, let $\Phi=\Phi(\hat{G},T)$ and let $\Phi_l$ be the set of long roots in $\Phi$.
If $\{ \alpha_1,\alpha_2,\alpha_3,\alpha_4\}$ is a basis of $\Phi$ then $\{ \alpha_2,\alpha_1,\alpha_2+2\alpha_3,\alpha_2+2\alpha_3+2\alpha_4\}$ is a basis of $\Phi_l$.
Note that $(\alpha_2+2\alpha_3)^\vee=\alpha_2^\vee+\alpha_3^\vee$ and $(\alpha_2+2\alpha_3+2\alpha_4)^\vee=\alpha_2^\vee+\alpha_3^\vee+\alpha_4^\vee$, thus the subgroup of $\hat{G}$ generated by $T$ and all $U_\alpha$ with $\alpha\in\Phi_l$ is isomorphic to $\Spin(8,k)=G$.
The subgroup $H$ of $\hat{G}$ generated by $G$ and $N_{\hat{G}}(T)$ normalizes $G$ and the corresponding map $H\rightarrow\Aut G$ is surjective.
(In fact $H$ is isomorphic to the semidirect product of $G$ by the symmetric group $S_3$.)

\begin{lemma}
All Kac automorphisms of type $D_4^{(3)}$ which have Kac coefficients equal to 0 or 1 are of positive rank, with the exception of the Kac automorphism of order 9.
These diagrams are:

(a) $111$, which has rank 1 and order 12 and acts as a Coxeter element of $F_4$,

(b) $101$, which has rank 2 and order 6 and corresponds to an element of type $F_4(a_1)$,

(c) $001$, which has rank 2 and order 3 and corresponds to an element of type $A_2\times\tilde{A_2}$,

(d) $010$, which has rank 1 and order 6 and corresponds to an element of type $C_3$,

(e) $100$, which has rank 1 and order 3 and corresponds to an element of type $\tilde{A_2}$ (this is the automorphism $\gamma$ constructed in \S 2).
\end{lemma}

\begin{proof}
Let $\hat{G}$, $G$ and $T$ be as above.
By assumption on $T$, any automorphism of $G$ is of the form $\Int n_w|_G$ for some $n_w\in N_{\hat{G}}(T)$.
Recall by Lemma \ref{0and1} that a Kac diagram which corresponds to a positive rank automorphism satisfies $h(\beta_i)\in\{ 0,1\}$ for $i=0,1,2$.
Moreover, there are no elements of $W(F_4)$ of order 9 and hence the only Kac diagrams which can be of positive rank are the five diagrams listed in the Lemma.
Clearly there is exactly one such Kac diagram of order 12.
Since $n_w^{12}=1$ if $n_w\in N_{\hat{G}}(T)$ represents a Coxeter element in $W(F_4)$, there exists at least one automorphism of rank 1 and order 12 and hence this automorphism has Kac diagram $111$.
Let us recall our description of the Kac automorphism corresponding to this class.
Let $\zeta$ be a primitive 12-th root of unity.
Then $\theta$ is the automorphism satisfying:
$$e_{\pm\alpha_1}\mapsto \zeta^{\pm 1} e_{\pm\alpha_3},e_{\pm\alpha_2}\mapsto \zeta^{\pm 1} e_{\pm\alpha_2},e_{\pm\alpha_3}\mapsto \zeta^{\pm 1} e_{\pm\alpha_4},e_{\pm\alpha_4}\mapsto \zeta^{\pm 1} e_{\pm\alpha_1}$$
With this clarification, we can determine which Kac diagrams correspond to $\theta^2$ and $\theta^4$.
(We note that $\theta^2$ and $\theta^4$ are necessarily of positive rank, since $\theta$ is.)
Indeed, the fixed point subspace for $\theta^2$ is spanned by $h_2$, $h_1+h_3+h_4$ and $e_{\pm(\alpha_1+\alpha_2)}+\zeta^{\pm 2} e_{\pm(\alpha_2+\alpha_3)}-\zeta^{\pm 4} e_{\pm(\alpha_2+\alpha_4)}$.
Thus $\theta^2$, which corresponds to an element of $W(F_4)$ of type $F_4(a_2)$, has Kac diagram $101$.
Similarly, the fixed point subspace for $\theta^4$ is spanned by $h_2$, $h_1+h_3+h_4$, $e_{\pm\alpha_1}+\zeta^{\pm 4} e_{\pm\alpha_3}+\zeta^{\mp 4}e_{\pm\alpha_4}$, $e_{\pm(\alpha_1+\alpha_2)}+\zeta^{\pm 2} e_{\pm(\alpha_2+\alpha_3)}-\zeta^{\pm 4} e_{\pm (\alpha_2+\alpha_4)}$ and $e_{\pm(\alpha_1+\alpha_2+\alpha_3)}+e_{\pm(\alpha_2+\alpha_3+\alpha_4)}+e_{\pm(\alpha_1+\alpha_2+\alpha_4)}$.
Thus the dimension of the fixed-point space is 6, and hence the corresponding Kac diagram is $100$.

On the other hand, let $n_w\in N_{\hat{G}}(T)$ be such that $w=n_wT$ is an element of type $C_3$ (resp. $\tilde{A_2}$) and $\Int n_w$, as an automorphism of $\hat{G}$, is conjugate to a Kac automorphism with diagram $01010$ (resp. $10001$).
Then $\Int n_w|_G$ has order 6 (resp. 3) and has rank at least 1 (since $\Lie(T)\cap{\mathfrak g}(1)$ is non-trivial); but $\Int n_w$ has rank 1 (as an automorphism of $\Lie(\hat{G})$) by Lemma \ref{order6} and Corollary \ref{order3}.
Thus $\Int n_w|_G$ is a rank one automorphism, which must also be of type $\tilde{D_4}^{(3)}$ since elements of $W(F_4)$ of type $C_3$ and $\tilde{A_2}$ are of order 3 modulo $W(D_4)$.
Hence there exist automorphisms of order 3 and 6 and of rank 1, which must correspond to the remaining two Kac diagrams as indicated in the diagram.
\end{proof}

\section{Calculation of the Weyl group}

It remains to calculate the little Weyl group for each of the automorphisms of the previous section.
To begin we recall the following straightforward observation \cite[Lemma 4.2]{thetagroups}.
We maintain the assumptions on $\theta$, ${\mathfrak c}$, $T$ from the previous section.

\begin{lemma}\label{w1}
Let $\bar{W}=W^\theta/Z_{W^\theta}({\mathfrak c})$.
Then $\bar{W}$ acts on ${\mathfrak c}$ and $W_{\mathfrak c}$ is a subgroup of $\bar{W}$.
\end{lemma}

We will see that for all $\theta$-groups of type $G_2$, $F_4$ and $D_4^{(3)}$, $\bar{W}=W_{\mathfrak c}$.
This is not true in general, see Rk. \ref{finalremark}(c).

\begin{lemma}\label{ord}
Let $\Gamma$ be a product of irreducible admissible diagrams of order $m$ and let $w\in W$ be an element of type $\Gamma$.
Let $\Phi_1$ be the smallest root subsystem of $G$ containing all roots in $\Gamma$ and let $\Phi_2$ be the set of roots in $\Phi$ which are orthogonal to $\Phi_1$.
Let ${\mathfrak t}(1) =\{ t\in{\mathfrak t}\, |\,w(t)=\zeta t\}$ and let $W_0=\{ w\in Z_W(w)\, |\, w|_{{\mathfrak t}(1)}=1_{{\mathfrak t}(1)}\}$.
Then $W_0$ contains $W(\Phi_2)$.
\end{lemma}

\begin{proof}
This is clear from Lemma \ref{w0prep}.
\end{proof}

From now on, fix $w$ and let $W_1$ (resp. $W_2$) denote the subgroup of $W$ generated by all $s_\alpha$ with $\alpha\in\Phi_1$ (resp. $\Phi_2$).
Then Lemma \ref{ord} shows that the order of $\bar{W}$ (and hence of $W_{\mathfrak c}$) divides $\#(Z_W(w))/\#(W_2)$.
Use of this straightforward observation will allow us to identify $W_{\mathfrak c}$ for all the cases which concern us.
Let $T_i$, $i|m$ be the subtori of $T$ defined in Lemma \ref{stabletori}.
We will need the following lemma (\cite[Lemma 4.3]{thetagroups}).

\begin{lemma}\label{criterion}
Let $T'_m=\prod_{i\neq m}T_i=\{ t^{-1}\theta(t)\, |\, t\in T\}$.
Suppose $\{ t\in T_m=T(0)\, |\, t^m=1\}\subset T'_m$.
Assuming $G$ is simply-connected, $W_{\mathfrak c}=\bar{W}$.
\end{lemma}

Our first result is for maximal rank automorphisms.

\begin{lemma}\label{maxrank}
Suppose $\theta$ is a maximal rank automorphism.
Then $W_0$ is trivial.
Thus $\bar{W}=W_{\mathfrak c}=W^\theta$.
\end{lemma}

\begin{proof}
By assumption, $G$ is simply-connected.
But now $Z_G({\mathfrak c})$ is connected, and therefore equals $T$.
Thus $Z_{N_G(T)}({\mathfrak c})=Z_G({\mathfrak c})=T$.
This shows that $W_0$ is trivial.
Furthermore, $T'_m=T$ by assumption on $\theta$, thus any element of $\bar{W}$ has a representative in $G(0)$ by Lemma \ref{criterion}.
\end{proof}

Finally, we have the following preparatory lemma, which appeared in \cite{panslice} in characteristic zero, and was generalised to positive characteristic in \cite[Prop. 5.3]{thetagroups}.
Recall that a {\it KW-section} for $\theta$ is an affine linear subvariety ${\mathfrak v}\subset{\mathfrak g}(1)$ such that the restriction of the categorical quotient $\pi:{\mathfrak g}(1)\rightarrow{\mathfrak g}(1)\quot G(0)$ to ${\mathfrak v}$ is an isomorphism.

\begin{lemma}\label{Nreg}
Suppose $\theta$ is an N-regular automorphism.
Then $k[{\mathfrak t}]^W\rightarrow k[{\mathfrak c}]^{W_{\mathfrak c}}$ is surjective.
In particular, if $\theta$ is inner then the degrees of the generators of $k[{\mathfrak c}]^{W_{\mathfrak c}}$ are simply those degrees of the invariants of ${\mathfrak g}$ which are divisible by $m$.

Moreover, $\theta$ admits a KW-section.
\end{lemma}

We have the following application of Lemma \ref{maxrank}.
We identify automorphisms by Weyl group elements $w$: that is, $w$, where $\theta$ is conjugate to $\Int n_w$ as described in Sect. 4.
Let $G_t$ for the $t$-th group in the Shephard-Todd classification.

\begin{lemma}\label{maxrankweyl}
For the maximal rank automorphisms, the Weyl group is as described in the following list:

(a) {\it Type $G_2$.} Coxeter element: $W_{\mathfrak c}=\mu_6$; $A_2$: $W_{\mathfrak c}=\mu_6$;  $A_1\times \tilde{A_1}$: $W_{\mathfrak c}=W(G_2)$.

(b) {\it Type $F_4$.} Coxeter element: $W_{\mathfrak c}=\mu_{12}$; $B_4$: $W_{\mathfrak c}=\mu_{8}$; $F_4(a_1)$ or $A_2\times\tilde{A_2}$: $W_{\mathfrak c}=G_5$; $D_4(a_1)$: $W_{\mathfrak c}=G_8$; $A_1^4$: $W_{\mathfrak c}=W(F_4)$.

(c) {\it Type ${D}_4^{(3)}$.} Type $F_4$: $W_{\mathfrak c}=\mu_{4}$; $F_4(a_1)$ or $A_2\times\tilde{A_2}$: $W_{\mathfrak c}=G_4$.
\end{lemma}

\begin{proof}
For the rank 1 inner automorphisms, this follows on reading off the orders of conjugacy classes (and hence centralizers) in \cite{carter}.
Moreover, since the maximal rank involutions in type $G_2$ and $F_4$ lie in the centre of the Weyl group, it remains only to check type ${D}_4^{(3)}$ and classes $F_4(a_1)$, $A_2\times\tilde{A_2}$ and $D_4(a_1)$ in type $F_4$.

Let $W$ be a Weyl group of type $F_4$ and let $W_l$ be the subgroup of $W$ generated by all $s_\alpha$ with $\alpha$ a long root.
Then $W_l$ is isomorphic to the Weyl group of type $D_4$, and is a normal subgroup of $W$ of index 6.
For an element in class $F_4(a_1)$ or $A_2\times\tilde{A_2}$, the centralizer in $W$ has order 72 by \cite{carter}.
Moreover, the square of an element in class $F_4(a_1)$ is an element in class $A_2\times\tilde{A_2}$, thus the little Weyl group for these two cases is equal.
There are six cosets of $W_l$ in $W$ and the factor group is isomorphic to $S_3$.
It follows that if $w$ is an element of $W$ which has order 3 modulo $W_l$ then $Z_{W_l}(w)$ is a normal subgroup of $Z_W(w)$ of index 3.
In particular, $Z_{W_l}(w)$ has order 4 (resp. 24) if $w$ is of Coxeter type (resp. of type $F_4(a_1)$ or $A_2\times\tilde{A_2}$).
Thus it is clear that $Z_{W_l}(w)\cong\mu_4$ in the case where $w$ is a Coxeter element of type $F_4$.

We claim that if $w$ is of type $A_2\times \tilde{A_2}$ then $Z_{W_l}(w)$ is isomorphic to $G_4$.
Indeed, $\#(Z_{W_l}(w))=24$, and thus the only other possibilities are $G(6,3,2)$, $G(12,12,2)$ or a product of two cyclic groups \cite{clark-ewing}.
(We require coprimeness of the characteristic here, which is automatic since we assume $\charac k=0$ or $\charac k>3$.)
But if $w$ is the fourth power of a Coxeter element $w_0$ then $w_0^3\in Z_{W_l}(w)$ is non-central and thus $Z_{W_l}(w)$ cannot be commutative.
If we write $w$ as $w_1w_2=w_2w_1$, where $w_1$ is an element of type $A_2$ and $w_2$ is an element of type $\tilde{A_2}$ then we can construct a basis $\{Êc_1,c_2\}$ for ${\mathfrak c}$ such that $w_i(c_j)=\zeta^{\delta_{ij}}c_j$.
Moreover, $w_1\in W_l$ and hence there exists an element of $Z_{W_l}(w)$ with characteristic polynomial $(t-\zeta)(t-1)$ as an automorphism of ${\mathfrak c}$.
Since there is no such element of $G(6,3,2)$ or $G(12,12,2)$, $Z_{W_l}(w)$ is isomorphic to $G_4$.
This proves the remaining cases in type $D_4^{(3)}$.
But now $Z_W(w)$ is a non-commutative pseudoreflection group of rank 2 which has polynomial generators of degree 6 and 12 by Lemma \ref{Nreg}, and thus is either $G(6,1,2)$, $G(12,4,2)$ or $G_5$.
Since $Z_W(w)$ contains $G_4$ as a normal subgroup, it follows that it is isomorphic to $G_5$.

For an element of type $D_4(a_1)$, the centralizer in $W$ has order 96 and has degrees 8 and 12 by Lemma \ref{Nreg}.
Thus the only possibilities for $Z_W(w)$ are $G(12,3,2)$, $G_8$, $G_{13}$ or $\mu_8\times\mu_{12}$ \cite{clark-ewing}.
Since $W_l$ is a normal subgroup of $W$, $\bar{W}$ contains a normal subgroup which is isomorphic to $G(4,2,2)$.
This rules out $\mu_8\times\mu_{12}$ and $G(12,3,2)$ since for example if $\xi$ is a primitive 12-th root of unity $$\begin{pmatrix}Ê\xi^2 & 0 \\Ê0 & \xi \end{pmatrix}\begin{pmatrix} 0 & 1 \\Ê1 & 0 \end{pmatrix} \begin{pmatrix}Ê\xi^{-2} & 0 \\Ê0 & \xi^{-1}Ê\end{pmatrix}=\begin{pmatrix} \xi & 0 \\Ê0 & \xi^{-1}Ê\end{pmatrix}\begin{pmatrix}Ê0 & 1 \\ 1 & 0 \end{pmatrix}$$ and hence $G(4,2,2)$ is not normal in $G(12,3,2)$.
Moreover, by \cite[p. 395]{cohen}, $G_{13}$ contains no reflections of order 4.
Thus $\bar{W}$ is equal to $G_8$.
\end{proof}

%In the following two lemmas we identify the remaining (positive rank) automorphisms in type $F_4$ by the action on a maximal torus whose Lie algebra contains ${\mathfrak c}$.
%Thus when we say for example that $w$ is of type $C_3$, we mean that $\theta$ is an automorphism with Kac diagram $01010$ (see Lemma \ref{coxpowers}).
For $w\in W$, let $\Phi_1$ be the smallest root subsystem of $\Phi$ containing all roots corresponding to vertices of the admissible diagram for $w$, let $\Phi_2$ be the set of roots in $\Phi$ which are orthogonal to all elements of $\Phi_1$ and let $W_i$ be the subgroup of $W$ generated by all $s_\alpha$ with $\alpha\in \Phi_i$ ($i=1,2$).
Let $L_1$ be the Levi subgroup of $G$ generated by $T$ and all $U_\alpha$ with $\alpha\in\Phi_1$.
(See the discussion in the paragraph preceding Lemma \ref{orthog}.)
Then $L_1$ is $\theta$-stable and ${\mathfrak c}\subset\Lie(L'_1)$, hence one can also consider the little Weyl group in $L_1$, which is naturally a subgroup of $W_{\mathfrak c}$.

\begin{lemma}
For the following automorphisms in type $F_4$, $\bar{W}=W_{\mathfrak c}$ is equal to the little Weyl group one obtains on restricting to the subgroup $L_1$.
We have:

(a) $W_{\mathfrak c}=\mu_6$ if $w$ is of type $C_3$ or $B_3$, 

(b) $W_{\mathfrak c}=\mu_4$ if $w$ is of type $B_2$.

(c) $W_{\mathfrak c}=\mu_2$ if $w$ is of type $\tilde{A}_1$.
\end{lemma}

\begin{proof}
This is a straightforward observation of the orders of the centralizer and the subgroup $W_2$ (see \cite[Table 8]{carter}).
If $w$ is of type $C_3$ or $B_3$ then $\Phi_2$ has a basis consisting of one element and hence $W_2=W(\Phi_2)$ has order 2.
But the centralizer of $w$ has order 12, thus the order of $\bar{W}$ divides 6.
If $w$ is of type $B_2$ then $\Phi_2$ is isomorphic to $B_2$ (hence $W_2$ has order 8) and the centralizer of $w$ has order 32, thus the order of $\bar{W}$ divides 4.
Finally, if $w$ is of type $\tilde{A}_1$ then $\Phi_2=B_3$ and hence the order of $W_2$ is 48.
Since the order of the centralizer is 96 by \cite{carter}, this shows that $\bar{W}=\mu_2$.
\end{proof}

We remark that in other types, $W_{\mathfrak c}$ may not be equal to the group one obtains on restricting to $L_1$.

\begin{lemma}\label{lastf4}
If $\theta=\Int n_w$ is an automorphism in type $F_4$ such that $w$ is of type $A_2$ or $\tilde{A_2}$ then $\bar{W}=W_{\mathfrak c}=\mu_6$.
In fact, there exists a $\theta$-stable semisimple subgroup $L$ of $G$ of type $B_3$ (if $w$ is of type $A_2$) or type $C_3$ (if $w$ is of type $\tilde{A}_2$) such that ${\mathfrak c}\subset\Lie(L)$, each element of $W_{\mathfrak c}$ has a representative in $L(0)$, and $\theta|_L$ is N-regular.
\end{lemma}

\begin{proof}
This is immediate since if $w$ is of type $A_2$ (resp. $\tilde{A_2}$) then $\theta$ is the square of an automorphism corresponding to a Weyl group element of type $B_3$ (resp. $C_3$).
\end{proof}

\begin{lemma}\label{dweyl}
Let $\theta$ be an automorphism of type $D_4^{(3)}$ with Kac diagram $010$ or $100$.
Then $W_{\mathfrak c}=\bar{W}=\mu_2$.
\end{lemma}

\begin{proof}
We noted in the proof of Lemma \ref{maxrankweyl} that if $w\in W=W(F_4)$ has order 3 modulo $W_l=W(D_4)$ then $Z_{W_l}(w)$ has index 3 in $Z_W(w)$.
Thus in both cases here the group $\bar{W}$ of Lemma \ref{w1} is isomorphic to $\mu_2$.
It therefore remains only to prove that the non-trivial element of $\bar{W}$ has a representative in $G(0)$.

Let $w$ be an element of $W$ of type $C_3$.
Then $(T^w)^\circ$ is of dimension 1 and is generated by the coroot of a long root element $\beta$.
The idea here is that we can centralize by $\beta^\vee(k^\times)$ and obtain a Levi subgroup of $G$ whose Lie algebra contains ${\mathfrak c}$, and which has a little Weyl group isomorphic to $\mu_2$.
It is easy to see that $Z_G(\beta^\vee(k^\times))$ is of type $A_1\times A_1\times A_1$.
Moreover, $w$ and $w^2$ permute the 3 subgroups of type $A_1$ transitively.
(If, for example, we take $w=s_2s_3s_4$ in $W(F_4)$, then $\Phi_l$ has basis $\{\alpha_2,\alpha_1,\alpha_2+2\alpha_3,\alpha_2+2\alpha_3+2\alpha_4\}$; $\beta=\hat\alpha$, the longest root in $F_4$, and hence the roots in $\Phi_l$ which are orthogonal to $\beta$ are $\pm\alpha_2$, $\pm(\alpha_2+2\alpha_3)$ and $\pm(\alpha_2+2\alpha_3+2\alpha_4)$; furthermore, $w(\alpha_2)=\alpha_2+2\alpha_3$, $w(\alpha_2+2\alpha_3)=\alpha_2+2\alpha_3+2\alpha_4$ and $w(\alpha_2+2\alpha_3+2\alpha_4)=-\alpha_2$.)
%Moreover, if $w$ is an element of $W(F_4)$ of type $C_3$ then $(T^w)^\circ$ is generated by the coroot of a long root element $\beta$, $Z_G(\beta^\vee(k^\times))$ is of type $A_1\times A_1\times A_1$ and $w$ and $w^2$ permute the 3 subgroups of type $A_1$ transitively.
Setting $L=Z_G(\beta^\vee(k^\times))'$, it is thus easy to see that ${\mathfrak c}\subset\Lie(L)$ and that $N_{L(0)}({\mathfrak c})/Z_{L(0)}({\mathfrak c})\cong\mu_2$.
(Here the Kac automorphism with diagram $010$ restricts to an automorphism of $L$ of the form $(g_1,g_2,g_3)\mapsto ({^t}g_3^{-1},g_1,g_2)$; the Kac automorphism with diagram $100$ is conjugate to the square of the Kac automorphism with diagram $010$.)
Thus $\bar{W}=W_{\mathfrak c}=\mu_2$ in either case.
\end{proof}

\begin{rk}\label{finalremark}
(a) In the case where the ground field has characteristic zero and $G(0)$ is semisimple (arbitrary $G$), the rank and little Weyl group were determined by Vinberg in \cite{vin}.
It was shown in \cite[Prop. 18]{vin} that $G(0)$ is semisimple if and only if the corresponding Kac diagram has exactly one non-zero entry, which is equal to 1.
Our calculations for $W_{\mathfrak c}$ agree with \cite{vin} in these cases.

(b) Let $G_Z^\theta=\{Êg\in G\, |\, g^{-1}\theta(g)\in Z(G)\}$ and let $W^Z_{\mathfrak c}=N_{G_Z^\theta}({\mathfrak c})/Z_{G_Z^\theta}({\mathfrak c})$.
Then $W^Z_{\mathfrak c}$ is a subgroup of $\bar{W}$.
Recall that $\theta$ is {\it saturated} if $W_{\mathfrak c}=W^Z_{\mathfrak c}$.
It is immediate that all automorphisms in type $F_4$ and $G_2$ are saturated since in both cases the centre is trivial.
Moreover, it is not difficult to show without using our classification that any automorphism of type $\tilde{D}_4^{(3)}$ is saturated.

(c) There is a strong relationship between Vinberg's construction of $W_{\mathfrak c}$ and work of Brou\'e and Malle constructing certain pseudo-reflection groups in finite groups of (exceptional) Lie type \cite{broue-malle}.
In general, the group constructed by Brou\'e and Malle corresponds to our $\bar{W}$.
It is possible for $W_{\mathfrak c}$ to be a proper subgroup of $\bar{W}$.
For example, if $\theta=\Int n_w$ in type $E_6$, where $w$ is an element of type $D_4(a_1)$, then $W_{\mathfrak c}$ is either $G_8$ or $G(4,1,2)$.

(d) In Table 3 we have indicated the action of $\theta|_L$ for the cases above using the automorphism $\tau:\SL(2,k)^3\rightarrow\SL(2,k)^3$, $(g_1,g_2,g_3)\mapsto ({^t}(g_3)^{-1},g_1,g_2)$.
\end{rk}

\begin{theorem}\label{main}
Any $\theta$-group of type $G_2$, $F_4$ or $D_4^{(3)}$ has a KW-section.
\end{theorem}

\begin{proof}
To prove this we will observe that there exists a $\theta$-stable reductive subgroup $L$ of $G$ such that ${\mathfrak c}\subset\Lie(L)$, $N_{L(0)}({\mathfrak c})/Z_{L(0)}({\mathfrak c})=W_{\mathfrak c}$ and $\theta|_L$ is N-regular.
Then we can apply Lemma \ref{Nreg}.
Indeed, $L$ is simply the group $L_1$ (see the discussion before Lemma \ref{orthog}) in all cases except automorphisms of type $A_2$ or $\tilde{A_2}$ in type $F_4$ or those of type $C_3$ or $\tilde{A_2}$ in type $D_4^{(3)}$.
In case $A_2$ (resp. $\tilde{A_2}$) in type $F_4$ we can reduce to a group of type $B_3$ (resp. $C_3$) by (the proof of) Lemma \ref{lastf4}.
In cases $C_3$ and $\tilde{A_2}$ in type $D_4^{(3)}$ we can reduce to a subgroup of $G$ isomorphic to $\SL(2,k)^3$ as indicated in the proof of Lemma \ref{dweyl}.
\end{proof}

This establishes Popov's conjecture in all types other than type $E$.
It is possible to use similar methods to solve the remaining cases, but the calculations required would make the task very time-consuming.
An alternative approach to the problem (for inner automorphisms) is to consider representatives $n_w$ of (suitable) elements $w\in W$, and to use computational methods to determine the Kac automorphism corresponding to $\Int n_w$ in each case.
We will return to this question in future work with Benedict Gross, Mark Reeder and Jiu-Kang Yu.

\begin{table}
\begin{center}
\caption{Positive rank automorphisms in type $G_2$}
\vline
\begin{tabular}{cccccc}
\hline
Kac diagram & $m$ & $w$ & $r$ & $W_{\mathfrak c}$ & $L$ \\
\hline
111 & 6 & $G_2$ & 1 & $\mu_6$ & N-reg. \\
011 & 3 & $A_2$ & 1 & $\mu_6$ & N-reg. \\
010 & 2 & $A_1\times\tilde{A_1}$ & 2 & $W(G_2)$ & N-reg. \\
\hline
\end{tabular}\vline
\end{center}
\end{table}\nopagebreak

\begin{table}
\begin{center}
\caption{Positive rank automorphisms in type $F_4$}
\vline
\begin{tabular}{ccccccc}
\hline
Kac diagram & $m$ & $w$ & $r$ & $W_{\mathfrak c}$ & $L$ & $\theta|_L$ \\
\hline
11111 & 12 & $F_4$ & 1 & $\mu_{12}$ & N-reg. &  \\
11101 & 8 & $B_4$ & 1 & $\mu_8$ & $\Spin(9)$ & Coxeter \\
10101 & 6 & $F_4(a_1)$ & 2 & $G_5$ & N-reg. & \\
01010 & 6 & $C_3$ & 1 & $\mu_6$ & $\Sp(6)$ & Coxeter \\
11100 & 6 & $B_3$ & 1 & $\mu_6$ & $\Spin(7)$ & Coxeter \\
10100 & 4 & $D_4(a_1)$ & 2 & $G_8$ & N-reg. & \\
01001 & 4 & $B_2$ & 1 & $\mu_4$ & $\Spin(5)$ & Coxeter \\
00100 & 3 & $A_2\times \tilde{A_2}$ & 2 & $G_5$ & N-reg. & \\
11000 & 3 & $A_2$ & 1 & $\mu_6$ & $\Spin(7)$ & positive 3-cycle \\
10001 & 3 & $\tilde{A_2}$ & 1 & $\mu_6$ & $\Sp(6)$ & positive 3-cycle \\
01000 & 2 & $A_1^4$ & 4 & $W(F_4)$ & N-reg. & \\
00001 & 2 & $\tilde{A_1}$ & 1 & $\mu_2$ & short $\SL(2)$ & Coxeter\\
\hline
\end{tabular}\vline
\end{center}
\end{table}\nopagebreak

\begin{table}
\begin{center}
\caption{Positive rank automorphisms in type $D_4^{(3)}$}
\vline
\begin{tabular}{ccccccc}
\hline
Kac diagram & $m$ & $w$ & $r$ & $W_{\mathfrak c}$ & $L$ & $\theta|_L$ \\
\hline
111 & 12 & $F_4$ & 1 & $\mu_{12}$ & N-reg. & \\
101 & 6 & $F_4(a_1)$ & 2 & $G_4$ & N-reg. & \\
010 & 6 & $C_3$ & 1 & $\mu_2$ & $\SL(2)^3$ & $\tau$ \\
001 & 3 & $A_2\times\tilde{A_2}$ & 2 & $G_4$ & N-reg. & \\
100 & 3 & $\tilde{A_2}$ & 1 & $\mu_2$ & $\SL(2)^3$ & $\tau^2$ \\
\hline
\end{tabular}\vline
\end{center}
\end{table}
%\nopagebreak

\bibliography{biblio}

\end{document}